\documentclass[amscd,amssymb,verbatim,11pt]{amsart}
\numberwithin{equation}{section}
\usepackage{epsfig}
\usepackage{multido}
\usepackage{color}
\usepackage{pstricks}
\usepackage{pst-all}
\usepackage{pst-math}
\usepackage{pst-func}
\usepackage{pst-3dplot}
\usepackage{graphicx}
\usepackage{amsmath}
\usepackage{a4,amssymb, tikz}

\usepackage[normalem]{ulem}

\newtheorem{thm}{Theorem}[section]
\newtheorem{cor}[thm]{Corollary}
\newtheorem{lem}[thm]{Lemma}

\theoremstyle{definition}

\usepackage{color}

\newif\ifShowLabels
\ShowLabelstrue
\newdimen\theight
\def\TeXref#1{
     \leavevmode\vadjust{\setbox0=\hbox{{\tt
            \quad\quad  {\small  \bf #1}}}%
     \theight=\ht0
     \advance\theight  by  \dp0
     \advance\theight  by  \lineskip
     \kern -\theight \vbox  to
     \theight{\rightline{\rlap{\box0}}%
      \vss}%
      }}%

    {\begin{thm}\label{#1} \ifShowLabels \TeXref{#1} \fi}%
    {\end{thm}}

    {\begin{def}\label{#1} \ifShowLabels \TeXref{#1} \fi}%
    {\end{def}}

    {\begin{lem}\label{#1} \ifShowLabels \TeXref{#1} \fi}%
    {\end{lem}}

    {\begin{cor}\label{#1} \ifShowLabels \TeXref{#1} \fi}%
    {\end{cor}}

\newcommand{\eqRef}[1]%
     {\ifShowLabels \TeXref{#1} \fi
      \begin{equation}\label{#1} }

\ShowLabelsfalse

\newcommand{\vsp}{\vskip 1em}
\newcommand{\vspp}{\vskip 2em}
\newcommand{\NI}{\noindent}
\newcommand{\bea}{\begin{eqnarray}}
\newcommand{\eea}{\end{eqnarray}}
\newcommand{\IR}{\mathbb{R}}
\newcommand{\bas}{\begin{align*}}
\newcommand{\eas}{\end{align*}}
\newcommand{\ba}{\begin{align}}
\newcommand{\ea}{\end{align}}

\newcommand{\be}{\begin{equation}}
\newcommand{\ee}{\end{equation}}
\newcommand{\ben}{\begin{eqnarray*}}
\newcommand{\een}{\end{eqnarray*}}
\newcommand{\lam}{\lambda}
\newcommand{\Om}{\Omega}
\newcommand{\om}{\omega}
\newcommand{\p}{\partial}
\newcommand{\al}{\alpha}
\newcommand{\bt}{\beta}

\newcommand{\Lam}{\Lambda}
\newcommand{\g}{\gamma}
\newcommand{\ve}{\varepsilon}
\newcommand{\dl}{\delta}
\newcommand{\D}{\Delta}

\newcommand{\G}{\Gamma}

\newcommand{\s}{\sigma}

\newcommand{\lamo}{\lambda_{\Om}}

\newcommand{\rh}{\rho}

\newcommand{\mP}{\mathcal{P}}

\newcommand{\mS}{\mathcal{S}}
\newcommand{\mQ}{\mathcal{Q}}

\newcommand{\tht}{\theta}

\title[Asymptotics]{A Strong Minimum principle and Large Time Asymptotics for viscosity solutions to a class of doubly nonlinear possibly degenerate parabolic equations}

\author[Bhattacharya and Marazzi]{Tilak Bhattacharya and Leonardo Marazzi}

\begin{document}

\maketitle

\begin{abstract} We study a version of the strong minimum principle, and large time asymptotics of positive viscosity solutions to classes of doubly nonlinear parabolic equations of the form 
$$ H(Du,D^2u)-u^{k-1}u_t=0,\;\;k\ge 1,\quad\mbox{in $\Om\times [0,T)$},$$ 
where $\Om\subset \IR^n$ is a bounded domain and $0<T\le \infty$. The spatial operator $H$ is homogeneous with power $k$. 
\end{abstract}

\section{\bf Introduction}

Let $\Om\subset \IR^n, \;n\ge 2$, be a bounded domain, and $\overline{\Om}$ be its closure.  For $0<T\le \infty$, define $\Om_T=\Om\times(0,T)$. If $T=\infty$, we write $\Om_\infty=\Om\times (0,\infty)$. 
Let  $P_T=(\overline{\Om}\times \{0\})\cup (\p\Om\times [0,T)$, and $P_\infty=P_T$ with $T=\infty$. These are the parabolic boundaries of
$\Om_T$ and $\Om_\infty$ respectively. 
Let $u=u(x,t)\;: \Om_T\rightarrow [0,\infty)$.  For $k\ge 1$, set
\eqRef{Gamma} 
\G_k[u]:=H(Du,D^2u)- u^{k-1}u_t,
\ee
where $H$ is an operator that satisfies certain conditions, described later in this section, and $k$ represents the homogeneity of $H$. 

We introduce notation for a more precise formulation of the questions studied in this work. The letters $x,\;y$ and $z$ will often denote points in $\IR^n$. We reserve $o$ for the origin. There will be occasions where we write $x=(x_1,x_2,\cdots, x_n)$. The notation $S^n$ is for the set of all $n\times n$ real symmetric matrices, $I$ is the identity matrix and $O$ is the zero $n\times n$ matrix. 
The letters $e$ and $\om$ will often stand for unit vectors in $\IR^n$. 

In this work, we study a version of the strong minimum principle and large time asymptotic behaviour of continuous positive viscosity solutions to the following doubly nonlinear parabolic equation
\eqRef{pde}
\G_k[u]=0,\;\;\mbox{in $\Om_T$ and $u=h$ on $P_T$,}
\ee
where $\G_k$ is as in (\ref{Gamma}) and $h=h(x,t)\in C(P_T)$. We allow $T=\infty$ in what follows.

The function $h=h(x,t),$ for $(x,t)\in P_T,$ comprises the initial and side conditions and is as given below:
\eqRef{h}
h(x,t)=\left\{ \begin{array}{ccc} h(x,0) & \forall x\in \overline{\Om},\;t=0,\\ h(x,t) & \forall (x,t)\in  \p\Om\times[0,T). \end{array}\right.
\ee
The function $h\in C(P_T)$ in the sense that 
\ben
&&\mbox{(i) } \quad\mbox{If $(x,t)\in \p\Om \times (0,T)$ and $(x,t)\rightarrow (y,0),\;y\in \p\Om$, \;\; then}\;\; \lim_{ (x,t)\rightarrow (y,0) } h(x,t)=h(y,0).\\
&&\mbox{(ii) } \quad \mbox {If $x \in \overline{\Om}$ and $x \rightarrow y,\;y\in \overline{\Om}$,\;\;then}\;\;\lim_{ (x,0)\rightarrow (y,0) } h(x,0)=h(y,0).
\een

We assume throughout that 
\eqRef{h(x,t)} 
0<\inf_{P_T} h(x,t)\le \sup_{P_T}h(x,t)<\infty.
\ee

\vsp

We describe now the conditions satisfied by $H$. These apply through out the work. 
\vsp
{\bf Condition A (Monotonicity):} Assume that $H:\IR^n\times S^n\rightarrow \IR$ is continuous. Moreover, 
we require that $H(\wp, O)=0$, for any $\wp\in \IR^n$. For any $X,\;Y\in S^n$ with $X\le Y$,
$$H(\wp,X)\le H(\wp,Y),\;\;\forall \wp\in \IR^n. \quad \Box$$
\vsp
{\bf Condition B (Homogeneity):} There is a constant $k_1\ge 0$ such that $\forall (\wp,X)\in \IR^n\times S^n$,
$$H(\tht \wp, X)=|\tht|^{k_1}H(\wp, X),\;\;\forall\; \tht\in \IR, \;\;\mbox{and}\;\;\;H(\wp, \tht X)=\tht H(\wp, X),\;\;\forall \; \tht>0.\quad \Box$$
Note that we do not assume that $H$ is odd in $X$. Also, if $k_1=0$ then $H(\wp, X)=H(X).$ 
\vsp
Set $k=k_1+1$. Clearly,
\eqRef{gamma}
H(\tht \wp, \tht X)= \tht^k H(\wp, X) \quad\mbox{where}\quad k=k_1+1\;\;\mbox{and}\;\;\tht>0.
\ee
\vsp
Let $\Lam\in\IR$ and $e\in \IR^n$ be a unit vector. Define
\bea\label{mM}
&&m(\Lam)=\min\left( \min_{|e|=1}H\left(e,I-\Lam e\otimes e\right),\; -\max_{|e|=1}H(e, \Lam e\otimes e- I)  \right),\;\;\mbox{and}    \nonumber \\
&&M(\Lam)=\max\left( \max_{|e|=1}H\left(e,I-\Lam e\otimes e\right), \;    - \min_{|e|=1}H(e, \Lam e\otimes e-I)  \right).
\eea
Clearly, $m(\Lam)\le M(\Lam)$. By Condition A, if $\Lam\le 1$ then $m(\Lam)\ge 0$, since $I-\Lam e\otimes e\ge 0$. However, if $\Lam>1$ then no definite statement can be made about $I-\Lam e\otimes e$. 
\vsp
{\bf Condition C (Coercivity):} We require that $H$ satisfy 
\eqRef{sec2.7}
\mbox{C(i)}\quad m(\Lam)>0,\;\forall \Lam<1,\quad \mbox{and}\quad \mbox{C(ii)}\quad M(\Lam)<0,\;\forall \Lam \ge \Lam_1, 
\ee
for some $\Lam_1\ge 1. \quad \Box$

 We make a simple observation. If $\Lam=0$ then C(i) implies that 
\be\label{C}
\mbox{(i)\quad $H(e,I)\ge m(0)>0$,\quad and\quad $H(e, -I)\le -m(0)<0.$}
\ee 

\vsp
One of the major origin of motivation for this work is the class of parabolic equations studied  in [\cite{ED0}: see Chap II].
In particular, we refer to equation (1.1) and the conditions in $(\text{A}_1),\;(\text{A}_2)$ and $(\text{A}_3)$ therein. The example of the parabolic equation
$$(*)\quad\text{div}(|Du|^{p-2}Du)+|Du|^p=u_t,\;p>1,$$
is included in \cite{ED0}. If we use the change of variables $v=e^u$ then we obtain the well-known doubly nonlinear parabolic equation
$$(**)\quad\text{div}(|Dv|^{p-2}Dv)=v^{p-2}v_t.$$
See Subsection 2.2 for more details. The operator $H(Du, D^2u):=\text{div}(|Du|^{p-2}Du)$ is quasilinear, $k=p-1$, and odd in second derivatives.
It is easy to see that Conditions A and B are satisfied,  if $p\ge 2$. Also,
$$H(e, I-\Lam e\otimes e)= (n+p-2)-(p-1)\Lam.$$
If $n\ge 2$, Condition C is satisfied. Thus, many of our results would hold for $(*)$, for $p\ge 2$, suitably modified by the application of the change of variable $v=e^u$.

 The monograph \cite{ED0} studies equations like $(*)$, in greater generality, and in the context of weak solutions. It contains significant results regarding local behaviour including regularity. Our context in our current work is is the setting of viscosity solutions. We study equations like $(**)$ in this context and also operators $H$ that may be fully nonlinear. Examples such as the Pucci operators, including their degenerate versions, are  instances included here.
While our focus is on equations of the type in $(**)$, we do utilize versions of the kind $(*)$ 
(in our context it would be $H(Du, D^2u+Du\otimes Du))$ to show that a version of the comparison principle holds.
In the context of viscosity solutions, this property has great utility. 
\vsp
Further examples of operators $H$ that satisfy Conditions A, B and C include, the pseudo $p-$Laplacian ($p\ge 2$), the infinity-Laplacian and the Pucci operators. See [\cite{BM4}: Section 3] and [\cite{BM5}: Section 3] for a detailed discussion.

As indicated above, doubly nonlinear parabolic equations are of great interest and there are many works that address them. The ones that are directly related to this work are \cite{BM1, BM2, BM20, BM3, BM5, TR}. The works in 
\cite{AJK, JL, Tych} are also of interest in this context. As discussed above, the work in \cite{ED0} has also close connections with this topic.

The results in this work are of two kind. The first addresses the strong minimum principle for $\G_k$. For $k=1$, we show that the results known 
for the linear case appear to hold even though $H$ may be fully nonlinear, see \cite{ED, L}. For $k>1$, however, there appears to be a departure from the linear case, as our results will show. Many of the results, known for $k=1$, fail to hold. 

In this context, it is well-known there is a close connection between the strong minimum principle and the Harnack inequality. The latter is known for many of the examples of $H$ listed above. However, we have been unable to provide a unified proof of a version that holds for the entire class of operators $\G_k$ being addressed in this work. To better appreciate the importance of Harnack's inequality, we direct the reader to the references \cite{ED0, ED, LSU, L}. The texts \cite{ED, L} address the linear case; \cite{ED0} develops techniques for studying classes of nonlinear, degenerate parabolic equations. These provide deep insights into the properties of solutions to such equations. A further expanded version of these topics can be found in \cite{EDUV}.

 The second set of results address large time asymptotic behaviour of positive solutions to (\ref{pde}) (equations like $(**)$ are included here). We provide a general result and follow it up by proving a result that applies to the case when the side condition $h$ is a constant. In the latter, the case $k=1$ appears to be different from $k>1$.  In this connection see also \cite{AJK, BM3, JL}.
\vsp
 We provide additional definitions. 
Also, from hereon, we take $k\ge 1$.
\vsp
Let $U\subset \IR^{n+1}$ be a domain. By $usc(lsc)(U)$, we mean the set of all upper semi-continuous (lower semi-continuous) functions defined on the set $U$.

We provide a definition of a viscosity solution of 
\begin{equation}\label{G_k}
\G_k[u]\equiv H(Du,D^2u)-u^{k-1}u_t=0,\quad\mbox{in $\Om_T$\quad and}\quad u=h\quad \mbox{on $P_T$}.    
\end{equation}

A function $u\in usc(\Om_T),\;u>0,$ is said to be a viscosity sub-solution of the differential equation in (\ref{G_k}) in the set $\Om_T$ (or solves $\G_k[u]\ge 0$ in $\Om_T$), if, for any $\psi$, $C^2$ in $x$ and $C^1$ in $t$, such that  $u-\psi$ has a maximum at some point $(y,t)\in \Om_T$, we have
$$H(D\psi, D^2\psi)(y,t)-u(y,t)^{k-1}\psi_t(y,t)\ge 0.$$
We say $u$ is a sub-solution of the problem in (\ref{G_k}), if $u\in usc(\Om_T\cup P_T)$, $\G_k[u]\ge 0$ in $\Om_T$, and $u\le h$ on $P_T$. 

Similarly, $u\in lsc(\Om_T),\;u>0,$ is said to be a viscosity super-solution of the differential equation in (\ref{G_k}) in $\Om_T$ (or solves $\G_k[u]\le 0$, in $\Om_T$), if, for any $\psi$, $C^2$ in $x$ and $C^1$ in $t$, such that $u-\psi$ has a minimum at some $(y,t)\in \Om_T$, we have
$$H(D\psi,D^2\psi)(y,t)-u(y,t)^{k-1} \psi_t(y,t) \le 0.$$
We say $u$ is a super-solution of the problem in (\ref{G_k}), if $u\in lsc(\Om_T\cup P_T),\;u>0,$ $\G_k[u]\le 0$ in $\Om_T$, and $u\ge h$ on $P_T$. 

A function $u\in C(\Om_T)$ is a solution of $\G_k[u]=0$ in $\Om_T$, if it is both a sub-solution and a super-solution.  Similarly, $u\in C(\Om_T\cup P_T)$ is
a solution of the problem in (\ref{G_k}), if it is both a sub-solution and a super-solution of (\ref{G_k}). The above definitions can be extended to the case $T=\infty.$
\vsp
We state next the main results. Theorems \ref{thm:1} and \ref{thm:1.0} address issues related to the strong minimum principle. Such results are well-known for many equations, see, for instance, \cite{ED, L}. Our goal is to provide a unified proof that works for a large class of equations and in the viscosity framework. 

 In order to clarify the role of the hypotheses, described earlier, the operator $H$ will be assumed to satisfy Conditions A, B and C unless otherwise mentioned.

For the rest of the work, $B_\rho(x)$ is the $\IR^n$ ball centered at $x\in \IR^n$ with radius $\rho.$
In what follows, a vector $\g\in \IR^{n+1}$, will sometimes be written as $\g=(\g_1,\cdots, \g_{n+1})=(\vec{\g}_n,\g_{n+1})$, where $\vec{\g}_n\in \IR^n$.

\begin{thm}\label{thm:1} Let $\Om\subset \IR^n$ be a bounded domain and $T>0$. Suppose that $k=1$. Assume that 
$u\in lsc(U_T),$ satisfies
$$\G_1[u]\equiv H(D^2 u)-u_t\le 0,\quad \mbox{in $\Om_T$}.$$
Set $m=\inf_{\Om_T} u$.\\

(a) {\bf Hopf Boundary principle.} Assume that $\p\Om\in C^2$. Let $(p,\tau)\in \p\Om\times(0,T)$ be such that $u(p,\tau)=m$ and $u(x,t)>m$, near $(p,\tau)$. 
Suppose that $\g\in \IR^{n+1}$ is such that $\g_{n+1}\le 0$,   $\vec{\g}_n$ is not tangential to $\p\Om$ and points towards the interior of $\Om$. Suppose that, for some $\tht_0>0$, small, 
$(p,\tau)+\tht \g \in \Om_T, \;\mbox{for every $\tht\in (0,\tht_0)$,}$
then
$$\liminf_{\tht\rightarrow 0^+}\frac{u(p+\tht \vec{\g}_n,\tau+\tht \g_{n+1})-u(p,\tau)}{\tht}>0.  $$

(b) {\bf Strong Minimum principle.} Let $(p,\tau)\in \Om_T$ be such that $u(p,\tau)>m$. Then $u(x,t)>m$ for every $(x,t)\in \Om\times (\tau,T).$
The result holds without any assumptions on the smoothness of $\p\Om$. As a consequence, we get that if $u(p,\tau)=m$ then $u(x,t)=m$ in $\Om\times (0,\tau).$
\end{thm}
There are no restrictions on the sign of $u$. A proof is presented in Section 3. The proof of part (b) is achieved by using slanted cylinders, see also \cite{L}. 
\vsp
We address $k>1$ in the next result.

\begin{thm}\label{thm:1.0} Let $\Om\subset \IR^n$ be a bounded domain and $T>0$. Suppose that $k>1$.
Assume that $u\in lsc(\Om_T),\;u>0,$ is super-solution, i.e., 
$$\G_k[u]\equiv  H(Du, D^2u)-u^{k-1}u_t\le 0\quad\mbox{in $\Om_T$}.$$
Set $m=\inf_{\Om_T} u$. The following hold.

(a) Suppose that $m>0$. If for some $(p,\tau)\in \Om_T$, $u(p,\tau)>m$ then there is a $\rho>0$ such that $u>m$ in the cylinder $B_\rho(p)\times [\tau, T)$. As a consequence, 
if $u(p,\tau)=m$ then $u(p,s)=m$ for all $0<s<\tau$. 

(b) Suppose that $m=0$ and $(p,\tau)\in \Om_T$ is such that $u(p,\tau)=0$. Assume that $u\in C(\Om_T)$. Then there is a sequence of points $\{(x_\ell, t_\ell)\}_{\ell=1}^\infty \subset \Om_T$, such that
$t_\ell<\tau$, $u(x_\ell, t_\ell)=0$ and $(x_\ell, t_\ell)\rightarrow (p,\tau).$ 
\end{thm}
A proof appears in Section 4. An example shows that the result in Part (a) can not be improved. Also, the Hopf boundary principle may not hold if $k>1$. 
\vsp
The next two results address large time asymptotic behaviour. See \cite{AJK, BM3, JL}. 

\begin{thm}\label{thm:1.1} Let  $\Om\subset \IR^n,\;n\ge 2,$ be a bounded domain, and $h\in C(P_\infty)$ satisfy \eqref{h} and \eqref{h(x,t)}. Suppose that $k\ge 1$. 

(a)  Let $u\in lsc(\Om_\infty\cup P_\infty),\;u> 0,$ be a super-solution to \eqref{pde}, i.e., $\G_k[u]\le 0$. Assume that $u=h$ on $\p\Om\times[T,\infty)$.
If $\nu_{\inf}=\lim_{t\rightarrow \infty} \left( \inf_{\p\Om\times[t, \infty)} h\right)$ then 
$$\lim_{t\rightarrow \infty}   \left(\inf_{\overline{\Om}\times[t, \infty)} u \right)=\nu_{\inf}.$$

(b)  Let $u\in usc(\Om_\infty\times P_\infty),\;u>0,$ be a sub-solution to \eqref{pde}, i.e., $\G_k[u]\ge 0$. Assume that $u=h$ on $\p\Om\times[T,\infty)$. If
$\mu_{\sup}=\lim_{t\rightarrow \infty} \left( \sup_{\p\Om\times[t, \infty)} h\right)$ then 
$$\lim_{t\rightarrow \infty}   \left(\sup_{\overline{\Om}\times[t, \infty)} u \right)=\mu_{\sup}.$$ 
\end{thm}
See Section 5 for a proof.
\vsp
The next result provides a slight refinement in the special case where $h\equiv$ constant. 
\begin{thm}\label{thm:1.2} 
Let $\Om$ be a bounded domain that satisfies an uniform outer ball condition. Suppose that, for some $\nu\in \IR$,  $h=\nu,$ on $\p\Om\times[T,\infty)$ for some $T\ge 0.$ 

For parts (a) and (b), assume that $\nu>0$ and $k>1$.  Assume that the sub(super)-solution $u=\nu$ on $\p\Om\times [T,\infty).$ The following holds for any $\al<1/(k-1)$.

(a) If $u>0$ is a subsolution then  $\displaystyle{   \lim_{t\rightarrow \infty} t^{\al} \left( \sup_{\Om\times [t,\infty)} u-\nu\right)=0.}$

(b) If $u>0$ is a supersolution then $\displaystyle{ \lim_{t\rightarrow \infty} t^{\al} \left( \nu-\inf_{\Om\times[t, \infty)} u\right)=0.}$
\vsp
(c) Suppose that either $\nu=0$ and $k\ge 1$, or $\nu>0$ and $k=1$. Let $\lamo$ be the first eigenvalue of $H$ on $\Om$.  If $u\ge 0$ is a sub-solution then
$$\lim_{t\rightarrow \infty} \left( \frac{\sup_{x\in\Om} \log u(x,t) }{t} \right)\le -\lamo.$$
\end{thm} 
The result does not appear to hold for super-solutions. For more details and a proof of the theorem, see Section 6.
\vsp
In this work, we do not address existence issues for the parabolic problems (\eqref{pde}). See [\cite{BM5}: Theorems 1.2 and 1.3] for a discussion of such issues, see also \cite{BM1, BM2}.
\vsp
\section{Preliminaries}
\vsp
We present some elementary calculations that will be useful in the work. Included here are also a few results related to a comparison principle for parabolic equations.

For the rest of the work, $B_\rho(x)$ is the $\IR^n$ ball centered at $x\in \IR^n$ with radius $\rho.$

\subsection{Radial functions} 
Let $z\in \IR^n$ and $r=|x-z|$. Suppose that $v(x)=v(r),\;r\ge 0,$ is $C^2$ in $r>0$. Set $e=(x-z)/r$, in $r>0$. Then for $x\ne z$, 
\eqRef{sec2.5}
H(Dv, D^2v)=H\left(v^{\prime}(r)e, \frac{v^{\prime}(r)}{r}\left(I-e\otimes e\right)+v^{\prime\prime}(r) e\otimes e \right),
\ee
where $I$ is the $n\times n$ identity matrix. 

If $v(r)=r^\al,\;\al>0,$ then
$$H(Dv, D^2v)= \al^k r^{\al k-(k+1)}H(e, I+(\al-2)e\otimes e).$$

\subsection{Change of variable formula} See Lemma 2.3 in [\cite{BM5}: Section 3] for a more general statement. This implies that if $u\in usc(lsc)(\Om_T),\;u>0$, satisfies
$$H(Du, D^2u)-u^{k-1}u_t\ge (\le) 0 \quad \mbox{in $\Om$,}$$ 
and $w=\log u$, then $w\in usc(lsc)(\Om_T)$ and  
$$H(Dw, D^2w+ Dw\otimes Dw)-w_t \ge (\le) 0\quad \mbox{in $\Om_T$.} $$
These are in the sense of viscosity. The elliptic counterpart appears in [\cite{BM4}: Section 5]. See also \cite{BM1, BM2}.
\vsp
\subsection{Parabolic Comparisons} We state versions of the comparison principle used in this work.  Note that in all the results stated here, $\Om\subset \IR^n$ is a bounded domain and $0<T<\infty$. However, many of these continue to hold for $T=\infty$, by letting $T\rightarrow \infty$.

We begin with  a well known general principle about sub-solutions that we state without proof.

\begin{lem}\label{sub} Suppose that $H$ satisfies Condition A. Let $u\in usc(\Om_T\cup P_T),\;u\ge 0,$ solve $\G_k[u]\ge 0$ in $\Om_T$, and $\ve \ge 0.$  Then
the function $u_\ve=\max\{u, \ve\}$ solves 
$$\G_k[u_\ve]\ge 0\quad\mbox{in $\Om_T$}.$$ 
Similarly, if $v$ is super-solution i.e., $\G_k[v]\le 0$ then $v_\ve=\min\{v, \ve\}$ is a super-solution.
\end{lem}

Suppose that $F:\IR^+\times  \IR^n\times S^n\rightarrow\IR$ is continuous and satisfies
\bea\label{sec2.1}
\mbox{$F(t, \wp, X)\le F(t, \wp, Y),\;\forall(t,\wp)\in  (0,T)\times \IR^n$, with $X\le Y$,}
\eea
\vsp

\begin{lem}\label{sec2.2}{(Comparison principle)} Let $F$ be as in (\ref{sec2.1}), and
$f:(0,T)\rightarrow (0,T)$ be continuous.  Suppose that $u\in usc(\Om_T\cup P_T)$ and $v\in lsc(\Om_T\cup P_T)$ satisfy
$$F(t,Du, D^2u+Du\otimes Du)- f(t)u_t\ge 0\;\;\mbox{and}\;\;F(t, Dv, D^2v+Dv\otimes Dv)-  f(t)v_t\le 0$$
 in $\Om_T.$ If $\sup_{P_T}v<\infty$ and $u\le v$ on $P_T$ then $u\le v$ in $\Om_T$.
\end{lem}
See [\cite{BM5}: Lemma 4.1, Section 4]. See also \cite{CIL}, for a more general result. We apply the above to obtain:

\begin{thm}\label{sec2.3}{(Comparison principle)} Let $H$ satisfy Conditions A and B. 
Suppose that $u\in usc(\Om_T\cup P_T),\;u>0,$ and $v\in lsc(\Om_T\cup P_T),\;v>0,$ satisfy
$$\G_k(u)\ge 0,\;\;\mbox{and}\;\;\G_k(v)\le 0,\;\;\mbox{in $\Om_T$}. $$

Let $k\ge 1$, then the following quotient type comparison result holds in $\Om_T$:
$$\frac{u}{v}\le \sup_{P_T} \left(\frac{u}{v}\right).$$

Additionally, if $k=1$ then $u-v\le \sup_{P_T}(u-v)$. This holds without any sign restrictions.
\end{thm}
A proof can be found in [\cite{BM5}: Theorem 4.3, Section 4] (in Theorem 4.3, take $\phi(t)=e^t$). The functions $\phi=\log u$ and $\psi=\log v$ satisfy the equations of the kind in Lemma \ref{sec2.2}. It follows that
$\phi-\psi\le \sup_{P_T}(\phi-\psi)$. The conclusion of Theorem \ref{sec2.3} follows. See \cite{BM1, BM2} for a related version. If $k=1$, Lemma \ref{sec2.2} may be applied directly to prove the claim. $\Box$

\vsp
We now present a version that is a slight extension of Theorem \ref{sec2.3}. 
\begin{lem}\label{min3}
 Let $u\in usc(\Om_T\cup P_T)$ and $v\in lsc(\Om_T\cup P_T)$. Assume $u\ge 0$ and $v>0$ in $\Om_T\cup P_T.$ Suppose that
$$\G_k[u]\ge 0,\quad\mbox{and}\quad \G_k[v]\le 0 \;\;\;\mbox{in $\Om_T$}.$$
If $v>0$ on $U_T$ then
$u/v \le \max_{P_T} \left( u/v\right).$ In particular, if $u=0$ on $P_T$, then $u\equiv 0$ in $\Om_T$.
\end{lem} 
\begin{proof} If $u>0$ on $P_T$ then the conclusion follows from Theorem \ref{sec2.3}.  Suppose that $u\ge 0$ on $\Om_T\cup P_T.$ Thus,
\be\label{min3=0}
0\le \max_{P_T}(u/v) <\infty.
\ee

 For a fixed, small $\ve>0$, set $u_\ve=\max\{u, \ve\}$. By Lemma \ref{sub}, $u_\ve$ is a sub-solution, $u_\ve\ge \ve$, and, hence, by Theorem \ref{sec2.3} and (\ref{min3-0}),
$$\frac{u}{v}\le \frac{u_\ve}{v}\le \sup_{P_T} \frac{u_\ve}{v}=\max\left\{ \sup_{ \{u\le \ve\}\cap P_T} \frac{\ve }{v},\; \sup_{\{u>\ve\}\cap P_T}\frac{u}{v}  \right\}\le\max\left\{ \sup_{ \{u\le \ve\}\cap P_T} \frac{\ve }{v},\; \sup_{P_T}\frac{u}{v}  \right\}.$$ 
If $\sup_{P_T}(u/v)>0$, we take $\ve$ small enough to conclude the result. If not, we let $\ve\rightarrow 0$ to conclude that $u=0$ in $\Om_T$.

\end{proof}

\begin{cor}\label{min3-0} Let $\bar{u}\in usc(\Om_T\cup P_T)$ and $\bar{v}\in lsc(\Om_T\cup P_T),\;\bar{v}>-\infty.$ Assume that $\inf_{\Om_T}\bar{u}>-\infty$ with, possibly, $\inf_{U_T}\bar{u}=-\infty$. If
$$H(D\bar{u}, D^2\bar{u}+D\bar{u}\otimes \bar{u})-\bar{u}_t\le 0\;\;\mbox{and}\;\;H(D\bar{v}, D^2\bar{v}+D\bar{v}\otimes \bar{v})-\bar{v}_t\ge 0,\;\;\mbox{in $\Om_T$,}$$ 
then, $\bar{u}-\bar{v}\le \max_{P_T}(\bar{u}-\bar{v}).$
\end{cor}
\begin{proof} For $\ve \in \IR$, $\bar{u}_\ve=\max\{ \bar{u},\ve\}$ is a sub-solution. Apply Subsection 2.2 and Lemma \ref{sec2.2}. \end{proof}
\vsp

\section{ Proof of Theorem \ref{thm:1}: $k=1$}
\vsp
We take $k=1$. Then $\G_k=\G_1$ and $u>0$ solves 
$$\G_1[u]\equiv H(D^2u)-u_t\le 0\quad\mbox{in $\Om_T$}.$$
Note that $H$ can be a fully nonlinear operator, as described in Section 1, see discussion following Condition C.

We recall for ease of reference the conditions satisfied by $H$.  From Condition A, $H=H(X), \; X\in S^n,$ is continuous and non-decreasing in $X$. Condition B imposes that 
\bea\label{min1}
H(\tht X)=\tht H(X),\;\forall \tht>0.
\eea 
In addition, $H$ satisfies Condition C, see (\ref{mM}) and (\ref{sec2.7}). From (\ref{sec2.7}) (C(ii)),  
\be\label{min2}
L(\Lam)\equiv\min_{|e|=1}H( \Lam e\otimes e -I)\ge -M(\Lam)>0,\;\;\forall \Lam\ge \Lam_1\ge 1.
\ee
\vsp
Let $z\in \IR^n$. Set $r=|x-z|$ and $e=(x-z)/r$. Suppose that $\phi(x,t)=\phi(r,t)$ is $C^2$ in $x$. From (\ref{sec2.5}), in $r>0$,
\be\label{min5}
 H(D^2 \phi)= H\left(\frac{\phi_r}{r} \left( I-e\otimes e \right)+\phi_{rr} e\otimes e \right).
\ee

We introduce additional notation. Given two points $x$ and $y$ in $\IR^n$, the vector $\vec{xy}$, in $\IR^n$, is the directed segment with initial point $x$ and terminal point $y$.
Set $\vec{x}=\vec{ox}$; then $\vec{xy}=\vec{y}-\vec{x}.$
\vsp

{\bf Proof of Theorem \ref{thm:1}.} Set $m\equiv \inf_{\Om_T} u$. Since $\G_1[u-m]\le 0$, we may assume that $m=0$ and $u\ge 0$ in $\Om_T$.
\vsp
\subsection{\bf (a) Hopf Boundary Principle:} The proof is a slight modification of the standard proof. We provide details.

Let $(p,\tau)\in P_T,\;\tau<T,$ be such that $u(p,\tau)=0$. Suppose that there is a $\IR^{n+1}$ neighborhood $N$ of $(p,\tau)$ such that $u>0$ in $N \cap\Om_T$. 
\vsp
{\bf Side boundary:} Since $\p\Om$ has interior  $\IR^n$ ball property at $p$, there is a $0<\rho_0< \tau$ and  $q=q(\rho)\in \Om$, such that for $0<\rho\le \rho_0$, $$B_\rho(q)\subset \Om\quad \mbox{and}\quad p\in \p B_\rho(q)\cap \p\Om.$$
We take $\rho>0$, small enough so that $B_\rho(q)\times[\tau-\rho, \tau]\subset N$, and, hence,
\be\label{minH2}
u>0\quad\mbox{in}\quad B_\rho(q)\times[\tau-\rho, \tau].
\ee
\vsp
 Set $r=|x-q|$, and the $\IR^{n+1}$ partial spherical shell
$$  S=\left\{\;(x,t)\;:\;      \frac{\rho^2}{4} \le r^2+(\tau-t)^2 \le \rho^2,\;\;\mbox{and}\;\;\tau-\rho/4\le t\le \tau   \right\}.$$
Let $S_{in}$ and $S_{out}$ be the inner and the outer boundaries respectively. 
Clearly, 
\ben
&&S_{in}=\{(x,t)\;:\; r^2+(\tau-t)^2=\rho^2/4,\;\;\tau-\rho/4\le t\le \tau\} \subset B_{\rho/2}(q)\times[\tau-\rho/4, \tau],\\
&&S_{out}=\{ (x,t) \;:\;r^2+(\tau-t)^2=\rho^2,\;\;\tau-\rho/4\le t\le \tau\}\subset B_\rho(q)\times[\tau-\rho/4, \tau].
\een

\begin{figure}
\begin{tikzpicture}
   \draw (5,0) [solid]  arc (0:180:50mm and 5mm);
   \draw (-5,0) [solid]  arc (180:360:50mm and 5mm);
    \draw (2,0) [solid]  arc (0:180:20mm and 2mm);
   \draw (-2,0) [solid]  arc (180:360:20mm and 2mm);
   \draw (4,-1) [densely dashed]  arc (0:180:40mm and 4.5mm);
   \draw (-4,-1) [solid]  arc (180:360:40mm and 4.5mm);
   \draw (1.2,-1) [densely dashed]  arc (0:180:12mm and 1.5mm);
   \draw (-1.2,-1) [solid]  arc (180:360:12mm and 1.5mm);
   \draw (-5,0.7)--(0,0.7);
    \draw (-5.1,.6)--(-4.9,0.8);
     \draw (-.1,.6)--(.1,0.8);
     \draw (-2.5,.7) node[above] {$\rho$};
     \draw (2,0.7)--(0,.7);
      \draw (2.1,.8)--(1.9,.6);
      \draw (1,.7) node[above] {${\rho}/{2}$};
   \draw (0,-1) node[below] {$(q,\tau-\rho/4)$};
   \draw (0,0) node[above] {$(q,\tau)$};
    \draw (0,0.15) node[below] {$.$};
    \draw (0,-.85) node[below] {$.$};
    \draw (-5,0) .. controls (-4.7,-1/2) and (-4.7,-1/2) .. (-4,-1);
    \draw (4,-1) .. controls (4.7,-1/2) and (4.7,-1/2) .. (5,0);
   \draw (-2,0) .. controls (-1.7,-1/2) and (-1.7,-1/2) .. (-1.2,-1);
     \draw (1.2,-1) .. controls (1.7,-1/2) and (1.7,-1/2) .. (2,0);
\end{tikzpicture}
 \caption{The spherical shell $S.$}
 \end{figure}
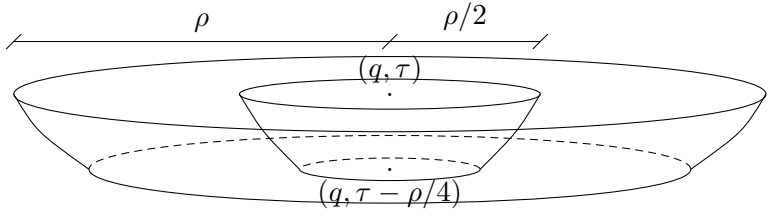

Set $U=(B_\rho(q)\setminus \overline{B}_{\rho/2}(q) )\times [\tau-\rho/4, \tau]$. Clearly, $S_{out}\subset U$ and $S_{in}$ is outside $U$. 
\vsp
We observe that the intersection of the $\IR^n$ plane $t=\tau-\rho/4$ with: 
\ben
&&\mbox{(i) $S_{out}$ is the $\IR^n$ sphere of radius $\sqrt{15}\rho/4$, and}\\
&&\mbox{(ii) $S_{in}$ is the $\IR^n$ sphere of radius $\sqrt{3}\rho/4$.}
\een
Thus, 
$$r\ge \sqrt{3}\rho/4,\;\; \mbox{in $S$.}$$

Moreover, recalling that $U\subset N$, by (\ref{minH2}), there is an $\ve>0$ such that
\be\label{minH4}
u(x,t)>\ve,\quad \mbox{if}\;\;0\le r\le \frac{\sqrt{15}\rho}{4}\;\;\mbox{and}\;\;\tau-\frac{\rho}{4}\le t\le \tau.
\ee
\vsp
We construct an auxiliary function  $\psi$ in $S$ as follows:
\be\label{minHs}
\eta(r,t)=\exp(-a(r^2+(\tau-t)^2),\;\;\mbox{and}\;\;\psi(x,t):=\psi(r,t)=\eta(r,t)-\eta(\rho, \tau),
\ee
where $a>0$ is to be determined. 

Observe that $\psi=0$ on $S_{out}$. Choose $a>0$. large, so that $\forall (x,t)\in S_{in}$, $0<\psi(x,t)=\exp(-a\rho^2/4)-\exp(-a \rho^2)<\ve.$ This ensures that
 $0<\psi\le \ve$ in $S$. Summarizing,
\be\label{S}
\mbox{$0<\psi\le \ve$, in $S$, $\psi=0$ on $S_{out}$ and $\psi\le \ve$ on $S_{in}.$} 
\ee 
Employing (\ref{min5}) and (\ref{minHs}), we obtain in $r\ge 0$,
$$H(D^2\psi)= H\left( 4a^2r^2\eta e\otimes e- 2a \eta I \right)
=2a \eta H(2ar^2 e\otimes e-I). $$
 We get from above,
\ben
\G_1[\psi]= H(D^2\psi)-\psi_t= 2a\eta \left[  H\left(2ar^2 e\otimes e-I \right) - (\tau-t) \right],\quad \mbox{in $S$}.
\een
Recalling (\ref{min2}) and that $r\ge \sqrt{3}\rho/4$, we choose $a$ large enough so that 
$$2ar^2\ge \frac{ 3a\rho^2}{8}\ge \Lam\ge \Lam_1.$$
Noting further that $\tau-t\le \rho/4$, (\ref{min2}) leads to
\ben
\G_1[\psi]\ge 2a\eta \left[ L(\Lam) - \frac{\rho}{4}  \right].
\een
By choosing $\rho$ small enough, $\psi>0$ is a sub-solution in $S$.

We now apply the comparison principle in Lemma \ref{min3}(see Theorem \ref{sec2.3}) in $U$. Note that $S_{out}\subset U$ and $S_{in}$ is outside $U$. Extending $\psi$ by zero in $r^2+(t-\tau)^2\ge \rho^2$, we get 
a sub-solution in $U$ (a proof is provided below). Next, recall (\ref{S}); $u\ge \psi$ on the parabolic boundary of $U$ since $u\ge \psi=0$ on $r=\rho$, $u>\ve\ge \psi$ on $r=\rho/2$ (see (\ref{minH4})), and $\psi\le \ve\le u$ on $t=\tau-\rho/4$. 
Thus, by the comparison principle,
$$\psi\le u \quad  \mbox{in $U$,}$$
 and, hence, in $S$.
We note that one could also apply the comparison principle directly to $S_{\rho,q}$ as $S_{out}$, $S_{in}$ and the flat base form its parabolic boundary. 

Recall (\ref{minHs}). Since $u(p,\tau)=\psi(p,\tau)=0$, for any $(x,t)\in S$, 
\bea\label{minH6}
u(x,t)-u(p, \tau) &\ge& \psi(x,t)-\psi(p, \tau)=\exp(-a \rho^2)\left\{ \exp\left( a [\rho^2-r^2-(\tau-t)^2] \right)-1 \right \}  \nonumber\\
&\ge& \exp(-a \rho^2)\left[ a \left\{ \rho^2-r^2-(\tau-t) ^2 \right\} \right].
\eea

Let $\vec{\g}\in \IR^{n+1}$ and $\tht_0$ be as in the statement of the lemma. Choose $(x,t)=(p,\tau)+\tht \vec{\g}=( p+\tht \vec{\g}_n, \tau+\tht \g_{n+1}).$ Note that 
$\langle \vec{pq}, \vec{\g}_n \rangle>0$ and $\g_{n+1}\le 0$.

Since $\vec{px}=\tht \vec{\gamma}_n$ and $\vec{xq}= \vec{pq}-\vec{px} =\vec{pq}-\tht \vec{\g}_n$, 
\ben
\rho^2-r^2\ge (\rho-r)\rho\ge \rho\left( \rho- |\vec{xq}| \;\right)\ge c\rho \tht | \vec{\g}_n|,
\een
where $c=c(\vec{\g}_n, \vec{pq})>0.$
Clearly, (\ref{minH6}) implies that for any $0<\tht\le \tht_0$, 
\ben
\frac{u(x,t)}{|(x,t)-(p,\tau)|}=\frac{u(x,t)-u(p,\tau)}{ \tht | \vec{\g}| }  \ge  a\exp(-a\rho^2)  \left( \frac{ c\tht \rho | \vec{\g}_n|-\tht^2|\g_{n+1}|^2}{\tht |\vec{\g}|} \right) >0,
\een
if $\tht_0$ is small enough.
\vsp
Finally, we check that $\psi$ is a sub-solution. It is enough to check the definition at points on $S_{out}$. Suppose that 
$\phi$, $C^2$ in $x$ and $C^1$ in $t$, is such that $\psi-\phi$ has a maximum at a point $(y,s)\in S_{out}$, i.e., $(\psi-\phi)(x,t)\le (\psi-\phi)(y,s)$. Since $\psi\ge 0$ and $\psi(y,s)=0$,  we get
that $\phi(y,s)\le \phi(x,t)-\psi(x,t)\le \phi(x,t).$ Thus, $\phi$ has a minimum at $(y,s)$ and so $D\phi(y,s)=0,\;\phi_t(y,s)=0$, and $D^2\phi(y,s)\ge 0.$ 
Thus, $H(D^2\phi(y,s))-\phi_t(y,s)\ge 0.\quad \Box$

\vsp
\begin{cor}\label{minH-1} Suppose that $u>0$ and $m>0$. By using $v=\log u$, the Hopf principle holds for
$$H(Dv, D^2v+Dv\otimes Dv)-v_t\le 0.$$
\end{cor}
See Subsection 2.1.
\vsp
\subsection{\bf (b) Strong Minimum Principle:} We continue to assume that $k=1$ and 
$$\G_1[u]=H(D^2u)-u_t\le 0\quad \mbox{in $\Om_T$}.$$

Suppose that $u\ge 0$ and $m=0$. Suppose that $u(p,\tau)>0$, for some $(p,\tau)\in \Om_T$.
\vsp
We make an observation. Let $\phi=\phi(x,t)\in C^2,\;\phi>0$ and $\bt\ge 2$. Then
\be\label{minS0}
 H(D^2\phi^\bt)= \bt \phi^{\bt-1}  H\left(D^2\phi + \left( \frac{\bt-1} {\phi} \right) D\phi\otimes D \phi \right).
\ee
\vsp
Let $(q, s)$ in $\Om_T$ with $s>\tau$. We comment on $q$ and $s$ later. Set 
$$\dl=s-\tau,\quad \vec{\g}_n=\vec{pq},\quad\mbox{and}\quad \vec{\g}=(\vec{\g}_n, \dl).$$
Clearly, $\vec{q}=\vec{p}+\vec{\g}_n.$ Vectors with lower case letters are in $\IR^n$, except for $\g$. 

The points $\vec{\mP}_t$ on the segment $S$ (in $\IR^{n+1}$) with end points $(p,\tau)$ and $(q,s)$, are parametrized by $t$ as
\be\label{minS2}
\vec{ \mP}_t:=\left(   \vec{p}  +  \left( \frac{t-\tau}{\dl} \right) \vec{\g}_n, \; t \right),\;\; \tau\le t\le \tau+\dl.
\ee
The notation $\vec{\mP}_t$ is a vector in $\IR^{n+1}$.

We call
$$  \vec{d}(x,t)=(\vec{x},t)-\vec{\mP_t}= \vec{x} -  \left( \vec{p}  +  \left( \frac{t-\tau}{\dl} \right) \vec{\g}_n  \right)\quad \mbox{and}\quad d(x,t)=|\vec{d}(x,t)|. $$
We will often write $d$ and $\vec{d}$ in place of $d(x,t)$ and $\vec{d}(x,t)$ respectively.  

Let $0<\D\le \dl.$ Define the slanted cylinder
$$C_{\rho, \D}=C_{\rho,\D}(\vec{p},\tau)=\left \{ (x,t)\; : \; d(x,t)\le \rho,\; \tau\le t\le \tau+\D \right\}.$$
Its axis is along the segment $S$, see (\ref{minS2}). Also, at $t=\tau+\D$, the point in $S$ is
$$\vec{\mP}_{\tau+\D}=\left( \vec{p}+\frac{ \D}{\dl} \vec{\g}_n, \tau+\D \right).$$ 
\vsp
Define 
\be\label{phi}
\phi(d)=\phi(x,t)= \rho^2 - d(x,t)^2 \;\;\;\mbox{and}\;\;\eta(t)=\frac{\tau+2\D-t}{2\D}.
\ee
Choose $0<\rho\le 1$, and set
\be\label{minS2a}
\psi(x,t)=\phi(d)^2\eta(t),\;\;\;\mbox{in $C_{\rho,\D}(p, \tau)$.}
\ee
\vsp
We show that for an appropriate $\D>0$ and $|\vec{\g}_n|\ne 0$, $\psi$ is a sub-solution in  $C_{\rho, \D}$. 
\vsp
We compute $H(D^2\psi)-\psi_t$. Recalling (\ref{minS0}),
\ben
 H(D^2\psi)=\eta H(D^2\phi^2) = 2 \eta \phi H\left(D^2\phi  + \frac{ D\phi\otimes D\phi }{\phi}\right).
\een

We note that
$$d^2=\langle\vec{d}, \vec{d}\rangle,\quad D\vec{d}=I\quad\mbox{and}\quad  \frac{\p \vec{d} }{\p t}=-\frac{ \vec{\g}_n }{\dl}.$$
Next, we write $\vec{d}=d\vec{e}=de$, where $e=e(x,t)$ is a unit vector. 
 Differentiating (\ref{phi}), 
$$D\phi= -2  \vec{d} ,\quad D^2\phi= -2I\quad\mbox{and}\quad D\phi\otimes D\phi= 4 \vec{d}\otimes \vec{d}=4d^2e\otimes e.$$ 
 Hence, from (\ref{minS2a}),
\ben
&& H(D^2\psi)= 4 \phi\eta \; H\left(  -I + \frac{ 2 d^2}{\phi} e\otimes e \; \right),\quad\mbox{and}\\
&&\psi_t=- \frac{ \phi^2}{2\D} + 4 \eta \phi \frac{\langle \vec{d}, \vec{\g}_n \rangle }{ \dl}= \phi  \left( \frac{ 4\eta \langle \vec{\g}_n, \vec{d}\rangle } {\dl} - \frac{\phi}{2\D}  \right).
\een 
Combining the two expressions, we get
\ben
 \G_1[\psi]=H(D^2\psi)-\psi_t &=&4 \phi \eta H\left( -I + \frac{2  d^2}{\phi}e\otimes e \; \right) +
\ve \phi  \left( \frac{ \phi}{2\D}  - \frac{ 4\eta \langle \vec{d}, \vec{\g}_n \rangle }{\dl}  \right)  \nonumber\\
&=&   \phi \left[ 4  \eta  H\left( -I + \frac{2  d^2}{\phi}e\otimes e \; \right) +
 \frac{ \phi}{2\D}  - \frac{ 4\eta \langle \vec{d}, \vec{\g}_n \rangle }{\dl} \right ].
\een
Noting that $\eta\le 1$ in $C_{\rho, \D}$, a rearrangement leads to
\bea\label{minS4-0}
\G_1[\psi] &\ge&   \phi\left[    \frac{\phi }{2\D} +  4 \eta  H\left( -I + \frac{2  d^2}{\phi}e\otimes e \; \right) -  \frac{ 4  d |\vec{\g}_n| }{\dl}\right].
\eea
\vsp
Next, noting (\ref{min2}), we choose $\Lam\ge \Lam_1$ and set
\be\label{minS3-0}
\upsilon=\sqrt{ \frac{\Lam }{\Lam+2} },\quad\; \upsilon_0=\sqrt{1-\upsilon^2}=\sqrt{ \frac{2 }{\Lam+2} }, \;\;\;\mbox{and}\;\;\;\;R=\upsilon \rho.
\ee
Then, $\phi(R)=\upsilon_0^2\rho^2$, and 
\ben
\mbox{in $R\le d<\rho$:}\qquad   \frac{2d^2}{\phi(d)}\ge  \frac{2 R^2}{\phi(R)}=\frac{2\upsilon^2}{\upsilon_0^2}=\Lam.
\een
Recalling (\ref{min2}), 
\be\label{minS4}
H\left( \frac{ 2 d^2}{\phi} e\otimes e -I \right)\ge H(\Lam e\otimes e-I)\ge L(\Lam)>0,\;\;\forall \; R\le d<\rho.
\ee

\vsp
We divide $0\le d<\rho$ into two intervals: $0\le d\le R$ and $R\le d<\rho$, where $R$ is as in (\ref{minS3-0}). 
\vsp
Recalling Conditions $A$, (\ref{min2}) and (\ref{minS4}), we estimate,  for some $M< 0$ (see (\ref{C})),
\be\label{minS5a}
H\left(  -I + \frac{2 d^2}{\phi}e\otimes e\right)\ge   \left\{ \begin{array}{lcr} L(\Lam)>0,&& \mbox{in $R\le d<\rho$,}\\ H(-I)\ge -|M|, && \mbox{in $0\le d\le R$}. \end{array}\right.
\ee
\vsp
Next, we derive conditions under which $\psi$ is sub-solution in $C_{\rho, \D}$.
\vsp
{\bf Interval ($R\le d< \rho$):} We use $1/2\le \eta\le 1$, (\ref{minS4-0}) and (\ref{minS5a}) to obtain
\ben
\G_1[\psi] \ge \phi\left[   4 \eta  H\left( -I + \frac{2  d^2}{\phi}e\otimes e \; \right) -  \frac{ 4  d |\vec{\g}_n| }{\dl}\right]
\ge \phi\left[  2 L(\Lam) -  \frac{ 4 \rho |\vec{\g}_n| }{\dl}\right].
 \een
 in $\tau\le t\le \tau+\D.$ 
 
Then $\psi$ is a sub-solution in $R\le d<\rho$ and $\tau\le t\le \tau+\D$, if 
\be\label{minS6}
 |\vec{\g}_n|  \le \frac{\dl L(\Lam)}{2\rho}.
 \ee
\vsp
{\bf Interval ($0\le d\le R$):} We use the estimates $\phi(R)\le \phi(d)\le \rho^2$ and $1/2\le \eta\le 1$. From (\ref{minS4-0}), (\ref{minS3-0}) and (\ref{minS5a}), we obtain
\ben
\G_1[\psi]  \ge  \phi   \left[    \frac{\phi(R) }{2\D} +  4 \eta  H(-I ) -  \frac{ 4  \rho |\vec{\g}_n| }{\dl}\right]
 \ge \phi\left[  \frac{ \upsilon_0^2\rho^2}{2\D}  -4|M| - \frac{ 4 \rho |\vec{\g}_n| }{\dl}   \right]
 \een
 
First, we choose $\D$ and $\g_n$ such that 
\be\label{minS7}
\D= \frac{\upsilon_0^{2}\rho^2 }{16 |M| } \equiv K_1 \rho^{2}\quad\mbox{and}\quad |\vec{\g}_n|\le \frac{|M| \dl}{ \rho}.
\ee
where $K_1=K_1(M, \Lam).$ Next, using (\ref{minS6}) and (\ref{minS7}), we select
\bea\label{minS8}
0<|\vec{\g}_n| \le  \dl \left( \min\left \{  \frac{ L(\Lam)}{2\rho},\;   \frac{ |M|}{\rho}  \right \}\right)
=\frac{K_2 \dl}{\rho},
\eea
where $K_2=K_2(L,M, \Lam).$ With these selections, $\psi$ is a sub-solution in $C_{\rho, \D}$.  Note that $\D$ depends on $\rho$ but is independent of $\dl$. However, $|\vec{\g}_n|$ is dependent on $\rho$ and $\dl$.

From the estimate for $\vec{\g}_n$, it is clear that we can allow $\vec{pq}$ large by selecting $\rho$ small. However, this makes $\D$ small. Thus, iterations may be needed to 
reach the time level $s$. We present details of the argument below.  
\vsp
Observe that the strip $\Om_{\tau+\D}\setminus \Om_\tau=\Om\times [\tau, \tau+\D).$ Our goal is to show that if $u(p,\tau)>0$ then $u(q,t)>0$ for $t$ in $[\tau, \tau+\D)$. 

Let $\vec{\g_n}=\vec{pq}$. Noting (\ref{minS8}), choose $0<\rho<1$ so that 
$$u(x,\tau)\ge \frac{u(p,\tau)}{2}, 
\quad \mbox{ in $|x-p|\le \rho$,}\quad \mbox{and}\quad 0<|\vec{\g_n}|\le \frac{ K_2\dl}{\rho}.$$
Fix $\rho$. Recalling (\ref{minS2a}) and (\ref{minS7}), we choose 
$$\D=K_1 \rho^{2}\quad\mbox{and}\quad 
\hat{\psi}(x,t)=\frac{ u(p,\tau)}{2}\left( \frac{\psi(d,t)}{\rho^4}\right)=\frac{u(p,\tau)}{2} \left(\frac{\phi(d)^2}{\rho^4}\right)\eta(t).$$ 
Note that $\G_1[c\psi]=c\G_1[\psi]\ge 0,$ if $c>0$, see Condition B. Thus, $\G_1\hat{\psi}\ge 0.$

 Observe that $0\le \hat{\psi}(x,\tau)\le u(p.\tau)/2$, in $B_\rho(p)$, and $\psi=0$ along the slanted cylindrical side. 
Thus, $\hat{\psi}$, with the selections made above, is a sub-solution in $C_{\rho, \D}$, and, by an extension by zero, in $\Om_{\tau+\D}\setminus\Om_\tau$.
That this extension  of $\hat{\psi}$ results in a sub-solution follows closely the argument presented in Subsection 3.1.

Using the comparison principle in Theorem \ref{sec2.3}, we obtain that
$u\ge \hat{\psi}$ in $\Om_{\tau+\D}\setminus \Om_\tau$. If $\dl<\D$, then, noting (\ref{minS2}) and (\ref{minS2a}), and taking $s=\tau+\dl$
\be\label{minS10-0}
 u \left(q, s \right)\ge \hat{\psi}\left(q, s \right)=\hat{\psi}(0, \tau+\dl)=\frac{u(p,\tau)}{2} \left( \frac{2\D-\dl}{2\D} \right)>0.
\ee
The claim follows. 
\vsp
Suppose that $\D\le \dl.$ Let $j=2,3,\cdots,$ be such that $(j-1)\D\le \dl< j \D$.
We use a chain of $j$ slanted $\IR^{n+1}$ cylinders of the same $t$ height and shape, with their axes along the segment $S$, see (\ref{minS2}).  Let $\D_s<\D$ be such that $j\D_s>\dl$. This adjustment is needed as $u$ is $lsc(\Om_T)$ and the relation $u\ge \psi$ (see above) may not extend to 
$t=\tau+\D$. As a result, we use the level $t=\tau+\D_s$ to bound $u$ from below by $\psi$. This modification is applied at every step. However, at every step, the $t$-heights of the cylinders, in which we define the auxiliary functions, continues to be $\D$. 
\vsp
We use an iterative process and apply the comparison principle at every step. We describe the individual cylinders and the auxiliary functions. The quantities $\rho, \; \g_n,\;\D,\;\D_s$ and $\phi$ stay the same at every step. The function $\eta$ will change.
For $\ell=0,1,2,\cdots$, set
\ben
\D=K_1\rho^2,\quad \eta_\ell(t)=\frac{ \tau_\ell+2\D-t}{2\D}\qquad\mbox{and}\qquad  \phi(d)=\rho^2-d^2,
\een
where $d=|(\vec{x},t)-\mP_t|$ (see (\ref{minS2})), and the quantity $\tau_\ell$ is defined below. Set for $\ell=1,2,\cdots$,

\begin{figure}\label{fig2}

    \begin{tikzpicture}
     \draw (-1.8,.2)--(4.2,.2);
      \draw (1.3,0.2) node[above] {$( \vec q,\tau+\delta)$};
    \draw (0,-1.5) node[below] {$( \vec p,\tau)$};
    \filldraw [black] (0.5,-.75) circle (2pt);
    \draw (-2.5,-.75)--(3.5,-.75);
    \filldraw [black] (1.13,.2) circle (2pt);
      \draw (.5,-.75) node[above] {$(\vec{p}+\frac{ \D}{\dl} \vec{\g}_n,\tau+\Delta)$};
      \filldraw [black] (0,-1.5) circle (2pt);
      \draw[dashed,->] (0,-1.5)--(2,1.5);
      \draw (-3,-1.5)--(3,-1.5);
         
          \draw (4.75,1)--(2.7,-2);
            \draw (-1.2,1)--(-3.4,-2);
    \end{tikzpicture}
    \caption{Vector $( \vec p,\tau)$ to $( \vec q,\tau+\delta)$}
\end{figure}
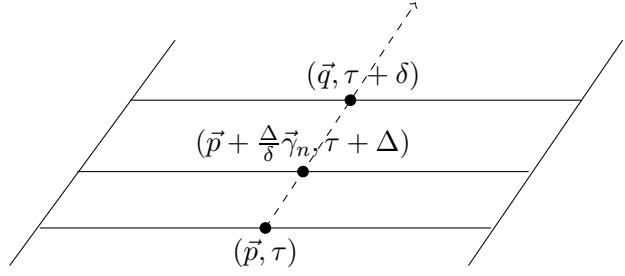

\bea\label{minS9}
&&\tau_0=\tau\quad \mbox{and}\quad \tau_{\ell}=\tau+\ell\D_s;\quad \vec{p}_0=\vec{p}\quad \mbox{and}\quad \vec{p}_\ell= \vec{p}+\frac{\ell \D_s}{\dl}\vec{\g}_n;   \nonumber\\
&&\vec{\mP}_0=(\vec{p},\tau)\;\;\mbox{and}\;\;
\vec{\mP}_\ell=\mP_{\tau_\ell}=\left(\vec{p}_\ell, \tau_\ell \right)
=\left( \vec{p}+\frac{\ell \D_s \vec{\g}_n}{\dl},\; \tau+\ell\D_s\right).
\eea

Next, we set
\bea\label{minS10}
\bt=\frac{2\D-\D_s}{2\D}\;\mbox{and}\;\;  \ve_\ell=\frac{\bt^\ell u(p,\tau)}{2};\;\; \hat{\psi}_\ell(x,t)=\frac{\ve_\ell \phi(d)^{2} \eta_\ell(t)}{\rho^4} \quad\mbox{in $C_{\rho,\D}(p_\ell,\tau_\ell)$.}
\eea

Applying the above procedure and taking $\ell=j$, we get that $u(q,s)\ge \hat{\psi}_j(q,s)$. Thus,
$$u(q,s)\ge \frac{\bt^j \phi(d)^{2}u(p,\tau)\eta_j(s)}{2\rho^{4}}\ge  \frac{u(p,\tau)\phi(d)^{2}}{2^{j+2}\rho^{4}}\ge
\frac{2 ^{-\dl/\Delta}u(p,\tau)\phi(R)^{2}}{4\rho^{4}}\ge 0,$$ since $\Delta\ge 1/2$, $\eta_\ell\ge 1/2$ and $\tau_{j-1}<s\le \tau_{j}.$ 

Next, suppose that $\mathcal{P}$ is a polygonal path in $\Om$ connecting $p$ to $q$. Let $\{(p_i\}_{i=0}^\ell,$ be the end points of the segments that comprise $\mathcal{P}$, and are such that $p_o=p$, $p_\ell=q$, and, for every $i$, $p_ip_{i+1}$ is a segment. Let $\g=\max\{|p_i-p_{i+1}|\}$.
We choose $\{\tau_i\}_{i=0}^\ell$ such that $\tau=\tau_0<\tau_1<\cdots<\tau_\ell=s $ and $\dl=\tau_{i+1}-\tau_i=(\tau-s)/\ell.$ Choose
$$\rho\le \max\left\{\frac{K_2\dl}{\g},\;\mbox{dist}\{\mathcal{P}, \p\Om\}\right\}.$$
Further adjust $\rho$ so that $u(x,\tau)\ge u(p,\tau)/2$ in $|x-p|\le \rho.$ 
Next, mount $\ell$ slanted cylinders each having width $\rho$ and axis along the $\IR^{n+1}$ segment with endpoints $(p_i,\tau_i)$and $(p_{i+1},\tau_{i+1})$. The value of $\Delta$ does not change and $|p_i-p_{i+1}|\le \g.$ We iterate the previously described process $\ell$ times to obtain an estimate $u(q,s)\ge c u(p,\tau)$, where $0<c=c(\ell, \dl,\Delta,\g)<1.$ 
The claim holds.
$\Box$

\vspp

\section{Proof of Theorem \ref{thm:1}: $k>1$}
Let $m=\inf_{\Om_T}u$. We show that the case $k>1$ differs quite markedly from $k=1$. This occurs even when $m>0$. In this case, our work appears to provide a complete description. However, things are not clear in the case $m=0$, and we provide what appears, to us, to be a partial result. One of the difficulties seems to be that the quotient version of the comparison principle (see Theorem 
\ref{sec2.3}) becomes unclear at places where both the sub-solution and the super-solution are small. 
\vsp
\subsection{\bf Part (a) $m>0$:}  If $u=m$ somewhere in $\Om_T$ then it appears that, in general, the strong minimum principle may fail to hold. The same appears to be the case for the Hopf boundary principle at points on $P_T$ where $u=m$. However, we show a weaker version of the strong minimum principle does hold. Before presenting the proof, 
we discuss an example that supports this assertion.
\vsp
{\bf Example:} Let $m>0$, $T>0$ and $k>1$. We construct a super-solution $\xi$, in an appropriate $\Om_T$, such that its minimum $m$ is attained along a $t$-segment $(p,t),\;0<t \le T$, for some $p\in \Om$, and some $T>0$.
However, $\xi>m$ in the rest of $\Om_T$.  Note that our construction produces a super-solution in $\IR^n\times (0, T).$
\vsp
 Set $r=|x|$ and $\phi(r)=r^{(k+1)/(k-1)}.$ Using (\ref{sec2.5}) (see Subsection 2.1) and (\ref{sec2.7}) i.e, Condition C(i),
\bea\label{minS12}
H(D\phi, D^2\phi)= c r^{(k+1)/(k-1)}H\left( e, I-\frac{k-3}{k-1} e\otimes e \right) 
\le c  \phi(r) L,
\eea
for some constant $c=c(k)>0$ and $L=\max_{|e|=1}H(e, I).$
\vsp
We take $\Om_T=B_R(o)\times [0,T)$, where $R>0$. Define
$$\xi(x,t)=m+ \phi(r) \eta(t),\;\mbox{where}\;\;\eta(t)=\left( \frac{ 1}{E (2T-t)} \right)^{1/(k-1)}\;\mbox{and}\;\; E=\frac{c (k-1) L }{ m^{k-1}}.$$

Note that 
$$\eta^{\prime}(t)=E \eta^k/(k-1)\ge 0.$$

Using (\ref{minS12}), we get in, $0<r<R$
\ben
\G_k[\xi]&=&H(D\xi, D^2\xi)-\xi^{k-1}\xi_t\le c \phi \eta^k  L -\left( m+\phi \eta \right)^{k-1} \phi \eta^{\prime}\\
&\le& \phi \eta^k\left[ c  L- \frac{ m^{k-1}E}{k-1}  \right]\le 0.
\een
\vsp
We verify below that $\xi$ is a super-solution in $\Om_T$. Observe that 
$$\xi(o, t)=m,\;\;0<t\le T,\;\;\;\mbox{and}\;\;\; \xi(x,t)>m, \;\;\;x\ne o.$$
This shows that $u$ does not attain its minimum value anywhere except along $(o,t),\; 0<t<T$.

Let $\nabla$ be the $\IR^{n+1}$ gradient. Then $\nabla \xi(o,t)=0,\;0<t<T$. We modify the example slightly to show that, in general, the Hopf boundary principle does not hold. Let $z\ne o$ and $\rho=|z|$. Consider the domain $U=B_{\rho}(z)\times [0, T]$ and $r=|x|$, as defined above. Thus, $\xi>m$ is a super-solution in $U$ and 
$\xi(o,t)=m,\;0<t<T$. This is a segment on the parabolic boundary of $U$. Since $\nabla \xi(o,t)=0$, the Hopf principle fails.  
\vsp
We now show that $\xi$ is a super-solution in $\Om_T$, Firstly, $\xi\ge m$ in $\overline{\Om}_T$ and $\xi(o,t)=m,\;\;0<t\le T.$  
It is sufficient to prove that $\xi$ is a super-solution at $r=0$. 

Let $\zeta$, $C^2$ in $x$ and $C^1$ in $t$, be such that $\xi-\zeta$ has a minimum at $(o,s)$ for some $0<s<T$. Then
$\xi(x,t)-\xi(o,s)\ge \zeta(x,t)-\zeta(o,s)$. Note that $\xi$ is $C^1$ in $x$, and $C^1$ in $t$. At $r=0$, we get that $D\xi(o,s)=D\zeta(o,s)=0$ and $\xi_t(o,s)=\zeta(o,s)=0$.  
Since $k>1$, we get that 
$$H(D\zeta(o,s), D^2\zeta(o,s) )-\xi(o,s)^{k-1} \zeta_t(o,s)=0.$$
This finishes the proof. $\Box$
\vsp
{\bf Proof of Part (a):} We now show that if $u>0$ satisfies $\G_k[u]\le 0$, in $\Om_T$, and $u(p,\tau)>m$, for some $(p,\tau)\in \Om_T$, then there is a cylinder $C\equiv B_\rho (p)\times [\tau, T)$ 
such that $u(x, t)>m$ in $C$. As a result, if $u(p,\tau)=m$ then $u(p, t)=m$ for all $0<t<\tau$.
As the above example shows, this result can not be improved.

Suppose that $u(p,\tau)>m$. Let $\ve>0$ and $0<\rho<1$ be such that 
$$u(x,\tau)\ge m+\ve,\quad\mbox{in $B_\rho(p)$}.$$

Set $\dl=T-\tau$, $r=|x-p|$ and $S=B_\rho(p)\times [\tau, T)$. Define in $S$
\be\label{minS10-0-1}
\psi(x,t)=m+\ve \phi(r)^2 \eta(t)\quad\mbox{where}\quad \phi(r)=\rho^2-r^2\quad\mbox{and}\quad \eta(t)=\frac{\tau+2\dl-t}{2\dl}.
\ee

Using (\ref{sec2.5}) and Condition B, we get
\ben
 H(D\psi, D^2\psi)&=&(\ve \eta)^k H\left(-4r\phi e, -4\phi(I- e\otimes e)+ (-4\phi+ 8r^2) e\otimes e \right)\\
&=&(4\ve \phi \eta)^k r^{k-1} H\left(e,  \frac{2r^2}{\phi}e\otimes e  -I \right).
\een 
Hence,
\be\label{minS11}
 H(D\psi, D^2\psi)-\psi^{k-1}\psi_t = \left( 4\ve\eta \phi\right)^k r^{k-1} H\left( e, \frac{2 r^2}{\phi}e\otimes e -I \; \right)  
+   \frac {\ve \psi^{k-1} \phi^{2} } {2 \dl}.
\ee
We now recall (\ref{minS3-0}) and (\ref{minS4}) and divide the interval $0\le r<\rho$ into the sub-interval $0\le r\le R$ and $R\le r<\rho$. As argued in (\ref{minS4}),
$\G_k[\psi]\ge 0,$  in $R\le r<\rho$ and $\tau<t<T$.

We consider 
$0\le r\le R$; use (\ref{minS5a}), and the three estimates: $1/2 \le \eta\le 1$, $\phi(r)\ge \phi(R)=(\upsilon_0\rho)^2$ and $m\le \psi\le m+\ve$, to obtain
\bea\label{minS12-1}
\G_k[\psi]  &\ge& -\left( 4 \ve\eta\phi \right)^k   r^{k-1} |M|
+   \frac {\ve \psi^{k-1}\phi^{2} } {2 \dl} 
\ge   \frac {\ve m^{k-1} \phi(R)^{2} } {2 \dl}   - \left( 4\ve \right)^k  \phi(0)^k  r^{k-1} |M|    \nonumber \\
&\ge & \ve\left [ \frac { m^{k-1} (\upsilon_0\rho)^4 } {2 \dl} - 4^k \ve^{k-1}  \rho^{3k-1} |M| \right]= \ve \rho^4 \left [ \frac { m^{k-1} \upsilon_0^4}{2\dl}-4^k \ve^{k-1}\rho^{3k-5}|M|\right].
\eea
If $\ve>0$ is small enough then $\psi$ is  a sub-solution in $B_\rho(p)\times [\tau,T)$.
\vsp
Next, we observe that $u\ge \psi=m$ on $\p B_\rho(p)\times [\tau, T)$ and $u(x,\tau)\ge m+\ve\ge \psi(x,\tau),$ for $x\in B_\rho(p)$. By using the comparison principle Theorem \ref{sec2.3}, we get 
$\psi\le u$ in $S$. Thus, for any $(x,t)\in S$ we get that
$$u(x,t)\ge \psi(x,t)=m+\ve (\rho^2-|x-p|^2)\left(  \frac{ T-t+ \dl }{2\dl} \right)>m.$$
The claim holds. $\Box$

\vsp

\subsection{\bf Part (b) $m=0$:} We consider the case where $u\ge 0$ and $m=\inf_{\Om_T}u=0$. We assume that $u\in C(\Om_T)$. 
\vsp
{\bf Proof of Part (b):} We show that the zeros are not isolated. Assume to the contrary. Let 
$C\equiv B_\rho(p)\times (\tau-\dl, \tau)\subset \Om_T$, for some $\rho>0$ and $\dl>0,$  such that $u>0$ in $\overline{C}\setminus \{(p,\tau)\}$.

Let $P$ be the parabolic boundary of $C$. Since $u>0$ on $P$, there is a $\nu>0$ such that $u\ge \nu$ on $P$.
Recall the calculations done in (\ref{minS10-0}), (\ref{minS12}) and (\ref{minS11}). Define in $S$,
$$\psi(x,t)=\frac{ \nu}{2}+ \ve (\rho^2-r^2)^2\left( \frac{ \tau-t}{2\dl}  \right),$$
where $0<\ve \le \nu/(2\rho^4)$. As done in Sub-section (\ref{minS12}), by choosing $\ve$ small enough, $\psi$ is a sub-solution in $C$. Moreover, $\psi \le\nu\le u$ on $P$. Hence, by Lemma \ref{min3} (see Theorem \ref{sec2.3}),
$u\ge \psi$ in $S$, and it is clear that by choosing points $(p,t),\;t<\tau$, close to $(p,\tau)$, $u(p,\tau)\ge \nu/2>0$, a contradiction. The claim holds. $\Box$
\vsp

\section{Proof of Theorem \ref{thm:1.1}}
\vsp
The proof generalizes the result Theorem 1.2 in \cite{BM3}, and is based on the use of auxiliary functions. We recall a few items and introduce two auxiliary functions before presenting the proof.

 We recall that $\Om_\infty=\Om\times (0,\infty)$ and $P_\infty=(\overline{\Om}\times\{0\})\cup (\p\Om\times (0,\infty))$.
For $t>0$, set
$$\mQ_t=\overline{\Om}\times [t,\infty)\quad \mbox{and}\quad \mS_t=\p\Om\times [t,\infty).
$$

Let $T>0$ be as in the statement of the theorem. We assume that $u=h$ on $\mS_T$.
Set
$$m=\min_{\mS_T} h\quad\mbox{and} \quad M=\sup_{\mS_T} h.$$
Thus, Theorem \ref{sec2.3} implies that
\bea\label{bnds}
&&\mbox{If $u>0$ is a sub-solution then}\; u\le \max\{\max_{\overline{\Om}} u(x,T),\;M\}\quad \mbox{in $\mQ_T$}, \nonumber\\
&&\mbox{If $u>0$ is a super-solution then}\; u\ge \min\{\min_{\overline{\Om}} u(x,T),\;m\}\quad \mbox{in $\mQ_T$}.
\eea
First, apply the comparison principle  in $\Om\times(T,s)$, for $s>T$, and then let $s\rightarrow \infty$ to get the claim.
\vsp
Recall the notation
$$\Gamma_k[w]:=H(Dw, D^2w)-w^{k-1}w_t.$$
\vsp
Let $z\in \IR^n\setminus \overline{\Om}$ and set $r=|x-z|$. In what follows, $D,\;E,\;F$ and $a$ are positive constants.  Our calculations, done next, show $a$ depends on $E$, see below. The constants 
$D,\;E$ and $F$ are chosen in the proof of the theorem. 

Set 
\be\label{thm1.1-2}
R=\sup_{x\in \Om}|x-z|,\quad \mathcal{R}=\inf_{x\in \Om} |x-z|,\quad \mbox{and}\quad\mathcal{D}=\mbox{diam}(\Om).
\ee
Clearly, $\mathcal{R}>0$, and $r\ge \mathcal{R}>0$, if $x\in \Om$. Also, 
$$R\le \mathcal{R+D}\quad\mbox{and}\quad \Om\subset B_{\mathcal{R}+\mathcal{D}}(z)\setminus B_{\mathcal{R}}(z).$$
\vsp

{\bf Auxiliary Function 1 (Sub-solution):}  Let $z$ and $r$ be as defined above. For constants $D,\;E, \; F$, and $a$, we define
the function $\xi\in C^2(\Om_\infty)$ as follows:
\be\label{thm1.1-3}
\xi(x,t)=\al(r) \tau(t),\;\;\;\mbox{where}\;\;\;\al(r)=D e^{Er^2}\;\;\mbox{and}\;\; \tau(t)=\frac{e^{at} }{e^{at}+F}.
\ee
\vsp
Thus,
$$\al^{\prime}(r)=(2Er) \al,\quad\al^{\prime\prime}(r)=2E\al\left( 1+2Er^2\right)\quad \mbox{and}\quad \tau^{\prime}(t)=\tau \left( \frac{aF }{e^{at}+F} \right).$$
Calling $\om=(x-z)/|x-z|$, we get
\ben
D\xi =2Er \xi\; \om,\quad\mbox{and}\quad \xi_t=\xi \left( \frac{aF }{ e^{at}+ F} \right). 
\een

Using (\ref{sec2.5}), we get 
\ben
D^2\xi =\tau \left[  \frac{ \al'}{r} \left( I-\om \otimes \om \right)+\al^{''} \om \otimes \om  \right]
=2E \xi \left( I+2E r^2   \om \otimes \om  \right).
\een
\vsp
Using the above observations and Conditions A, B and C, we get
\ben
\Gamma_k[\xi]&=&(2E\xi)^k r^{k-1} H(\om, I+ 2E r^2 \om\otimes \om )-   \xi^{k}\left( \frac{aF}{ e^{at}+ F} \right)    \\
&\ge &\xi^k \left[ (2E)^k r^{k-1}H(\om, I)-  \frac{aF}{ e^{at}+ F}\right].
\een
Recalling (\ref{thm1.1-2}),
\ben
\Gamma_k[\xi]\ge \xi^k \left[ (2E)^k \mathcal{R}^{k-1} H(\om, I)- a\right].
\een
Thus, $\xi$ is sub-solution in $\Om_\infty$ if we choose (see (\ref{sec2.7}) C(i))
\be\label{thm1.1-4}
0<a< (2E)^k \mathcal{R}^{k-1}\min_{|\om|=1} H(\om, I).  \quad \Box
\ee
\vsp
{\bf Auxiliary Function 2 (Super-solution):} Let $z$ and $r$ be as above. For positive constants $D,\;E,\;F$, and $a>0$, we set
\be\label{thm1.1-40}
\zeta(x,t)=\bt(r) \tht(t), \quad\mbox{where}\;\;\bt(r)=D e^{-Er^2}\quad \mbox{and}\quad \tht(t)=1+ F e^{-at}.
\ee
We impose a condition on $E$ and $a$ for $\zeta$ to be a super-solution. Rest are chosen in the proof of the theorem.
Clearly,
$$\bt^{\prime}=(-2Er) \bt,\;\; \bt^{\prime\prime}=2E\bt( 2Er^2-1),\;\;\mbox{and}\;\; \tht^{\prime}=-aF e^{-at}=-\tht \left( \frac{aFe^{-at} }{1+Fe^{-at} }\right). $$

Letting $\om=(x-z)/|x-z|$, we have
\ben
D\zeta&=&(-2Er)\zeta \;\om,\qquad \zeta_t=-\zeta  \left( \frac{ aF e^{-at} }{ 1+Fe^{-at} } \right),\\
D^2\zeta &=&\tht \left[  \frac{\bt'}{r} \left( I-\om\otimes \om \right)+    \bt^{''} \om \otimes \om    \right]\\
&=& 2E\zeta \left( 2Er^2 \om\otimes \om-I \right).
\een
Thus,
\bea\label{thm1.1-5}
\Gamma_k[\zeta] 
&=& H\left(  -2E\zeta r \;\om, 2E\zeta\left( 2Er^2 \om\otimes \om-I \right)   \right)+ \zeta^{k}  \left(  \frac{ aF e^{-at} }{1+Fe^{-at} }  \right)   \nonumber \\
&=& (2E\zeta)^kr^{k-1}  H\left( \om, 2Er^2 \om\otimes \om-I\right) + \zeta^k \left(  \frac{ aF e^{-at} }{1+Fe^{-at} }  \right)     \nonumber\\ 
&=& \zeta^k  \left[  (2E)^k  r^{k-1}H\left( \om, 2Er^2 \om\otimes \om-I\right) +a \left(  \frac{ F e^{-at} }{1+Fe^{-at} }  \right)  \right].
\eea

By (\ref{thm1.1-2}), $\mathcal{R}\le r\le \mathcal{R+D}.$ We choose $E$ (see (\ref{sec2.7}) C(i)) so that
$$0<\kappa\equiv2E(\mathcal{R}+\mathcal{D})^2<1\quad\mbox{and}\quad J\equiv\max_{|\om|=1}H(\om, \kappa\; \om\otimes \om-I)<0.$$
Next, select
\be\label{thm1.1-50}
0<a< (2E)^k \mathcal{R}^{k-1} \left | J \right|.
\ee
With the above choice for $E$ and $a$, we get
\ben
\Gamma_K[\zeta]\le  \zeta^{k} \left[ a + (2E)^k r^{k-1} H\left( \om, \kappa\; \om\otimes \om-I\right)  \right] 
 \le \zeta^{k} \left[ a-  (2E)^k \mathcal{R}^{k-1} |J| \right] \le 0.
\een
With these values, $\zeta$ is a super-solution in $\Om_\infty$. Note that $E$ depends on $\mathcal{R}.$ $\quad\Box$ 
\vsp
Let $t\ge T$. Define in $\mQ_t$ and $\mS_t$,
\bea\label{thm1.1-0}
&&\mbox{(i)}\;\mu_{\inf}(t)=\inf_{\overline{\mQ}_t} u,\;\;\mbox{(ii)}\;\mu_{\sup}(t)=\sup_{\overline{\mQ}_t} u,  \;\;
\mbox{(iii)}\;\nu_{\inf}(t)=\inf_{\mS_t} h,\;\;\mbox{and}\;\; \mbox{(iv)}\;\nu_{\sup}(t)=\sup_{\mS_t} h.
\eea
Since $u=h$ on $P_\infty$, $\mu_{\inf}(t)\le \nu_{\inf}(t)$, and $\nu_{\sup}(t)\le \mu_{\sup}(t).$
Set 
\be\label{thm1.1-1}
\nu_{\sup}=\lim_{t\rightarrow\infty} \nu_{\sup}(t)\;\;\;\mbox{and}\;\;\;\nu_{\inf}=\lim_{t\rightarrow\infty} \nu_{\inf}(t).
\ee

\vsp
{\bf Proof of Part (a) of Theorem \ref{thm:1.1}}:  Recall the notation in (\ref{thm1.1-0}), and (\ref{thm1.1-1}).
 We take $k\ge 1$. Recall that $u>0$ is a super-solution,  and since (\ref{h(x,t)}) holds, $\mu_{\inf}(t)<\infty,\;\forall\;t>0.$
\vsp
Note that $\mu_{\inf}(t)\le \nu_{\inf}(t).$ Thus, the claim follows if we show that
$$\lim_{t\rightarrow \infty} \mu_{\inf}(t)\ge \nu_{\inf}.$$
\vsp
Recall that $u=h$ on $\mS_T$ and $u\ge \min\{\min_\Om u(x,T),\;m\}\equiv m_0$. Since $ \nu_{\inf}\ge \mu_{\inf}(t) \ge m_0$, if $\nu_{\inf}=m_0$, the claim follows. Assume from here on that
$\nu_{\inf}>m_0.$

Let $\ve>0$ be small, and $T_0\ge T$, large, so that for $t\ge T_0$ (see (\ref{thm1.1-1}))
$$\nu_{\inf}(t) \ge \nu_{\inf}-\ve>m_0>0.$$

Fix $z\in \IR^n\setminus \Om$; set
$$r=|x-z|,\quad \mathcal{R}=\inf_{x\in \Om} |x-z|\quad \mbox{and}\quad \mathcal{D}=\mbox{diam}\;\Om.$$
 We employ Auxiliary Function 1, see (\ref{thm1.1-3}), and recall the condition (\ref{thm1.1-4}):
$$\xi(x,t)=D e^{Er^2} \left( \frac{e^{a(t-T_0)} }{e^{a(t-T_0)}+F} \right), \quad\mbox{where}\quad 0<a< (2E)^k \mathcal{R}^{k-1} \min_{|\om|=1} H(\om, I).  $$

We select
\bea\label{thm1.1-6}
D=m_0,\quad E=\frac{1}{(\mathcal{R+D})^2}\log\left( \frac{\nu_{\inf}-\ve}{m_0} \right),\quad \mbox{and}\quad F=\frac{\nu_{\inf}-\ve}{m_0}-1.
\eea
Observe that $e^{E(\mathcal{R+D})^2}=1+F=(\nu_{\inf}-\ve)/m_0$.

\vsp

Our aim is to show that $u\ge \xi$ in $\mQ_{T_0}$.
Use (\ref{thm1.1-6}) and that $\mathcal{R}\le r\le \mathcal{R+D}$. Thus,
\bea\label{thm1.1-7}
\frac{m_0 \;e^{E\mathcal{R}^2} }{ 1+F}\le \xi(x,T_0)\le \frac{m_0 \;e^{E(\mathcal{R+D})^2} }{ 1+F}=m_0\le u(x,T_0),\;\;\forall \;x\in\Om,
\eea
and
\bea\label{thm1.1-8}
&&\xi(x,t)\le m_0\; e^{E (\mathcal{R+D})^2 }\left( \frac{ e^{a(t-T_0)} }{  e^{a(t-T_0)}+F}\right)\le \nu_{\inf}-\ve\le h(x,t),\;\;\forall (x,t)\in \mS_{T_0}.
\eea

Employing the comparison principle, $u\ge \xi$ in $\mQ_{T_0}$. Using (\ref{thm1.1-6}), we have
\ben
u(x,t)&\ge& m_0\; e^{Er^2} \left( \frac{ e^{a(t-T_0)} }{ e^{a(t-T_0) } +F } \right)
= m_0\;   \left( \frac{ \nu_{\inf}-\ve}{ \bar{m} }   \right)^{r^2/(\mathcal{R+D})^2}\left( \frac{e^{a(t-T_0)} }{e^{a(t-T_0)}+F} \right)\\
&\ge& m_0 \;   \left( \frac{ \nu_{\inf}-\ve}{m_0}  \right)^{\mathcal{R}^2/(\mathcal{R+D})^2}\left( \frac{e^{a(t-T_0)} }{e^{a(t-T_0)}+F} \right),\;\;\forall(x,t)\in \mQ_{T_0}.
\een
\vsp
Since $u(x,t)\ge \mu_{\inf}(t)\ge \inf_{\mQ_t}\xi,\;t\ge T_0$, we get that
$$\mu_{\inf}(t)\ge m_0\;   \left( \frac{ \nu_{\inf}-\ve}{ m_0 }   \right)^{\mathcal{R}^2/(\mathcal{R+D})^2}\left( \frac{e^{a(t-T_0)} }{e^{a(t-T_0)}+F} \right). $$ 

Letting $t\rightarrow \infty$, and then letting $\mathcal{R}\rightarrow \infty$,
$$\lim_{t\rightarrow \infty} \mu_{\inf}(t)\ge \nu_{\inf}-\ve.$$
The claim follows since the above is true for any small $\ve$. $\Box$
\vsp
{\bf Proof of Part (b):} We assume that $u$ is a sub-solution. Recall that 
$M=\sup_{\mS_T} h(x,t).$  Set $M_0=\max\{u(x,T),\; M\}.$ 
As noted in (\ref{bnds}), $u(x,t) \le M_0$ in $\mQ_T$. Since $\nu_{\sup}\le \mu_{\sup}(t)\le M_0$, if $\nu_{\sup}=M_0$, the statement follows.  

For the proof, we assume that
$\nu_{\sup}<M_0$ and we will show that $\lim_{t\rightarrow \infty} \mu_{\sup}(t)\le \nu_{\sup}$. 
\vsp
Let $\ve>0$, small, and $T_0>0$ be such that
\be\label{thm1.1-8}
\nu_{\sup}\le \nu_{\sup}(t)\le \nu_{\sup}+\ve<M_0,\quad \mbox{for any $t\ge T$.}
\ee
This ensures that $h(x,t)\le \nu_{\sup}+\ve$ on $\mS_{T_0}.$
\vsp
We employ the function in (\ref{thm1.1-40}): let $z\in \IR^n\setminus \Om$ and $r=|x-z|$. Define 
$$\zeta(x,t)=\zeta(r,t)=D e^{-Er^2} \left( 1+F e^{-a(t-T_0)}\right),\;\;\forall (x,t)\in \mQ_{T_0},$$
where $D,\; E, \;F$ and $a$ are positive constants. Recalling (\ref{thm1.1-2})
and (\ref{thm1.1-50}), we choose
\bea\label{thm1.1-9}
&&0<a< (2E)^k   \mathcal{R}^{k-1}  |J|,\quad\mbox{where} \quad J=\max_{|\om|=1}H(\om, \kappa \om\otimes \om -I)<0, \nonumber\\
&&\mbox{and}\quad \kappa\equiv 2E(\mathcal{R+D})^2<1.
\eea
Choose $\kappa>0$, small ($E$ small), so that $J<0$ (see (\ref{sec2.7}) C(i)), and as a result, $\zeta$ is a super-solution in $\mQ_{T_0}$. 
\vsp
For a fixed $\kappa$, we choose
\ben
D=e^{\kappa/2}(\nu_{\sup}+\ve),\quad E=\frac{\kappa}{ 2(\mathcal{R+D})^2} \quad \mbox{and}\quad F=\frac{ M_0 }{\nu_{\sup}+\ve}-1.
\een
Thus, in $\mQ_{T_0}$, 
\ben
\zeta(x,t)=(\nu_{\sup}+\ve) \exp\left( \frac{\kappa}{2} \left[ 1-  \frac{ r^2}{(\mathcal{R+D})^2} \right] \right) \left( 1+ Fe^{-a(t-T_0)} \right).
\een
\vsp
Observe that if $x\in \Om$ then $\mathcal{R}\le r\le \mathcal{R+D}$. Hence, by (\ref{thm1.1-8}),
\ben
\zeta(x,T_0)&\ge& (\nu_{\sup}+\ve) \left( \frac{M_0}{\nu_{\sup}+\ve} \right)\ge M_0\ge u(x,T_0),\;\;\forall x\in \Om,\quad\mbox{and}\\
\zeta(x,t)&\ge&(\nu_{\sup}+\ve) \exp\left( \frac{\kappa}{2} \left[ 1-  \frac{ r^2}{(\mathcal{R+D})^2} \right] \right) \ge \nu_{\sup}+\ve\ge h(x,t),\;\;\forall(x,t)\in \mS_{T_0}.
\een

Thus, $\zeta\ge u$ on the parabolic boundary of $\mQ_{T_0}$, and Theorem \ref{sec2.3} implies that $\zeta\ge u$ in $\mQ_{T_0}$. Thus, for any $s\ge t>T_0,$
$u(x,s)\le \zeta(x,s)$, and
$$\mu_{\sup}(t)\le \sup_{\mQ_{t}}\zeta\le (\nu_{\sup}+\ve) \exp\left( \frac{\kappa}{2} \left[ 1-  \frac{ \mathcal{R}^2}{(\mathcal{R+D})^2} \right] \right) \left( 1+ Fe^{-a(t-T_0)} \right),\;\;\mbox{in $\mQ_t$},$$
for any $t>T_0$.

Let $t\rightarrow \infty$ and then let $\mathcal{R}\rightarrow \infty$ to obtain that $\lim_{t\rightarrow \infty}\mu_{\sup}(t)\le \nu_{\sup}+\ve$. The claim holds. $\Box$

\vsp
\section{Proof of Theorem \ref{thm:1.2}} 
\vsp
Before presenting the proof, we record the following. See Appendix A.1 for existence and comparison principles. 

\begin{lem}\label{thm1.2-3} Suppose that $\Om\subset \IR^n$ is a bounded domain that satisfies an outer ball condition. Let $k\ge 1$, $\dl\ne 0$ and $\tht\in \IR.$ Then there is a $\psi$ in $C(\overline{\Om})$ such that 
$$ H(D\psi, D^2\psi)=\dl,\quad\mbox{in $\Om$, with $ \psi=\tht$ on $\p\Om$.} $$
If $\dl>0$ then $\psi\le \tht$, and if $\dl<0$ then $\psi \ge \tht$.
Also, $\psi=\tht+ |\dl|^{1/k} \eta(x)$, where $H(D\eta, D^2\eta)=\dl/|\dl|$, and $\eta=0$ on $\p\Om$.
\end{lem}

{\bf Proof of Theorem \ref{thm:1.2}.} For parts (a) and (b), we assume that $\nu>0$ and $k>1$. Part (c) addresses the cases $\nu=0$ and $k\ge 1$. Also, we include a comment about $k=1$. 
\vsp
{\bf Proof of Part (a):} Assume that $u>0$ is a sub-solution and $u=\nu$ on $\mS_{T}$. 

Let $\ve>0$ be small. By Theorem \ref{thm:1.1}, there is a $T_0\ge T$ such that
\be\label{thm1.2-1}
 \nu\le \sup_{x\in \overline{\Om}}u(x,t)\le \nu+\ve,\quad\mbox{for any $t\ge T_0$.} 
\ee

By Lemma  \ref{thm1.2-3}, there is a function $\psi\ge 1$ in $C(\Om)$ such that
\be\label{thm1.2-2}
H(D\psi,D^2\psi)=-1\quad\mbox{in $\Om$ and $\psi=1$ on $\p\Om$.}
\ee 
Observe that $\psi\ge 1$ in $\Om$. 
\vsp
Let $T_1\ge T_0$, to be determined later. With $\psi$ as in (\ref{thm1.2-2}), set in $\mQ_{T_1}$,
$$\phi(x,t)=\nu+ \ve \psi(x) \tau(t)\quad\mbox{in $\mQ_{T_1}$,}\quad\mbox{where}\quad \tau(t)=\left( \frac{T_1}{t} \right)^{1/(k-1)}. $$
Define $M_\psi=\sup_{\overline{\Om}}\psi$. Clearly, 
\be\label{thm1.2-5}
1\le \psi\le M_\psi \quad\mbox{and}\quad \nu\le \phi \le \nu+ \ve M_\psi.
\ee
\vsp
Using that $\tau\le 1$, $\tau^{\prime}=-\tau/ [ (k-1) t ]$ and (\ref{thm1.2-2}), 
\ben
\G_k[\phi]=H(D\phi, D^2\phi)-\phi^{k-1}\phi_t&=& -\left[ \ve \tau \right]^k+  \phi^{k-1}     \left( \frac{\ve \psi}{k-1} \right) \left(\frac{\tau}{t}\right).
\een
Since $\tau^{k-1}=T_1/t$, using (\ref{thm1.2-5}), 
\ben
\G_k[\phi]= \ve \tau \left[  \frac{\psi\phi^{k-1} }{(k-1)t}  -[ \ve \tau]^{k-1}  \right] 
\le \frac{\ve \tau}{t}  \left[  \frac{ M_\psi\left( \nu+\ve M_\psi  \right)^{k-1}}{ k-1 } -  \ve^{k-1} T_1  \right].
\een

Hence, $\phi$ is super-solution in $\mQ_{T_1}$ if  
$$T_1\ge \max\left\{ \frac{M_\psi \left( \nu+ \ve M_\psi  \right)^{k-1}  } { (k-1) \ve^{k-1} },\;\;T_0 \right\}.$$

Next, from (\ref{thm1.2-1}) and (\ref{thm1.2-5}), 
$$u(x,T_1)\le \nu+\ve \le \phi(x, T_1)\;\;\mbox{and}\;\;u(x,t)=\nu\le \phi(x,t),\;\forall(x,t)\in \mathcal{S}_{T_1}. $$

By the comparison principle in Theorem \ref{sec2.3} and (\ref{thm1.2-1}),
$$\nu\le \sup_{\Om}u(x,t)\le \sup_{\Om}\phi(x,t)\le \nu+ \frac{\ve M_\psi T_1^{1/(k-1)} }{  t^{1/(k-1)} }=\nu+\frac{ K}{t^{1/(k-1)} }\quad\mbox{in $\mQ_{T_1}$,}$$  
where $K=K(k, \nu,T, M_\psi)$. Thus, 
$$\limsup_{t\rightarrow \infty}   \left[ t^{\al} \left(  \sup_{\Om}u(x,t)-\nu \right)  \right]= 0,\quad\mbox{for any}\;\;0<\al<\frac{1}{k-1}.$$
The claim holds. $\Box$
\vsp

{\bf Proof of Part (b):} We assume that $u>0$ is a super-solution.

In Lemma \ref{thm1.2-3}, take $\dl=1 $ and $\tht=-1$. Let $\psi$ be the solution.
 Set $M_\psi=\max_{\overline{\Om}} |\psi|;$ thus, 
$$-M_\psi\le \psi\le -1.$$

Define
\be\label{thm1.2-8}
T_\ve= \frac{\nu^{k-1}M_\psi} { (k-1)\ve^{k-1} },\quad\mbox{where}\;\; 0<\ve\le \ve_0\quad\mbox{and}\quad \nu-\ve_0 M_\psi>0.
\ee
Fix $0<\ve\le \ve_0$, small so that $T_\ve\ge T$. 

By Theorem \ref{thm:1.1}, let $T_0\ge T_\ve$ be such that
\be\label{thm1.2-60}
0<\nu-\ve\le \inf_{\overline{\Om}} u(x,t)\le \nu,\quad \forall (x,t)\in \mQ_{T_0}.
\ee

Set
$$\phi(x,t)=\nu+ \ve \psi(x) \left( \frac{T_0}{t} \right)^{1/(k-1)}=\nu- \ve |\psi(x)| \left( \frac{T_0}{t} \right)^{1/(k-1)},\quad \forall (x,t)\in \mQ_{T_0}.$$
By (\ref{thm1.2-8}), $0<\phi\le \nu$. Also, since $\psi\le -1$,
\be\label{thm1.2-70}
\phi(x,T_0)\le \nu-\ve,\;\;\mbox{in $\Om$, \;\;and}\;\;\;\phi(x,t)\le \nu\;\; \mbox{in $\mS_{T_0}$}.
\ee

Since $H(D\psi, D^2\psi)=1$, $\psi\le 0$, $0<\phi\le \nu$ and $T_0\ge T_\ve$, we have that
\ben
\G_k[\phi] &=& \ve^k \left( \frac{T_0}{t}\right)^{k/(k-1)} + \phi(x,t)^{k-1}    \left( \frac{\ve \psi}{k-1} \right) \frac{T_0^{1/(k-1) } }{ t^{k/(k-1)} }  \\
&\ge & \frac{\ve T_0^{1/(k-1)} }{ t^{k/(k-1)} }\left( \ve^{k-1} T_0- \frac{\nu^{k-1} M_\psi }{k-1}  \right)\ge 0.
\een
The last line follows from (\ref{thm1.2-8}).
\vsp
Since $\phi$ is sub-solution in $\mQ_{T_0}$ and, by (\ref{thm1.2-5}), $u\ge \phi$ on its parabolic boundary, using Theorem \ref{sec2.3}, we obtain that
$$u(x,t)\ge \phi(x,t)=  \nu+ \ve \psi(x) \left( \frac{T_0}{t} \right)^{1/(k-1)},\;\;\;\forall (x,t)\in \mQ_{T_0}.$$
Observe that $\inf_{\Om}\phi(x,t)\le \inf_{\Om} u(x,t)\le \nu$. 

If $0<\s<1/(k-1)$ we have
\ben
\liminf_{t\rightarrow \infty} \left[  t^{\s}\left( \inf_{\overline{\Om}} u(x,t) -\nu \right)  \right]= 0.
\een 
This proves the claim.  $\Box$

\vsp
{\bf Comment:} Let $\nu>0$ and $k=1$. Since $H(e, X)=H(X)$, $u$ is a sub-solution of
$$ H(D^2u)-u_t=0,\quad\mbox{in $\Om_\infty$}.$$ 
Clearly, $v\equiv u-\nu$  is a sub-solution and $v=0$ on $\Om\times [T,\infty)$. Parts (a) and (b) do not apply as $k=1$. The decay rate of $v$ turns out be exponential in $t$. See Part (c) below. $\Box$
\vsp
{\bf Proof of Part (c):} Let $\nu=0$ and $k\ge 1.$ We continue to assume that $u\ge 0$ in $\Om_\infty$. 
Let $T_0\ge 0$ be such that $h(x,t)=0$ on $\mS_{T_0}$. Define
$$\mathcal{M}=\sup_{\overline{\Om}} u(x,T_0).$$

We refer to Appendix A.2 for details,and in particular the definition of $\lamo$. Choose $\lam<\lamo$, close to $\lamo$. Let $\psi_\lam=\psi_\lam(x)>0$ solve
\be\label{thm1.2-6}
H(D\psi_\lam, D^2\psi_\lam)+\lam \psi_\lam^{k}=0\;\;\mbox{in $\Om$ and $\psi_\lam=\mathcal{M}$ on $\p\Om$.}
\ee
Set
$$\phi_\lam(x,t)=e^{-\lam (t-T_0)} \psi_\lam(x)\;\;\;\mbox{in $\mQ_{T_0}$}.$$

Since $\psi_\lam\ge \mathcal{M}$,  
$$\phi_\lam(x,T_0)\ge u(x,T_0),\;\;\forall x\in \Om\quad\mbox{and}\quad \phi_\lam(x,t)>u(x,t),\;\;\forall(x,t)\in \mS_{T_0}.$$
Also, (\ref{thm1.2-6}) yields
$$H(D\phi_\lam,D^2\phi_\lam)-\phi_\lam^{k-1} (\phi_\lam)_t =e^{-\lam k (t-T_0)}  \psi_\lam^{k} (   \lam -  \lam)=0.$$

The comparison principle in Lemma \ref{min3} in $\mQ_{T_0}$ implies that for any $T>T_0$,
$$0\le u(x,t)\le \phi_\lam= \psi_\lam(x) e^{-\lam (t-T_0)}\;\;\;\;\mbox{in $\Om\times (T_0,T)$.}$$
Clearly, the above holds in any large $T$ and so the estimate holds in $\mQ_{T_0}$. Thus, for any $t\ge T_0$,
$$\sup_{\Om} u(x,t)\le \max_{\Om} \psi_\lam(x) e^{-\lam(t-T_0)}.$$ 
Applying logarithm to both sides and letting $t\rightarrow \infty$, we obtain
$$\lim_{t\rightarrow \infty} \left( \frac{\sup_{\Om} \log u }{t} \right)\le -\lam.$$
The statement in the theorem now holds as $\lam<\lamo$ is arbitrary. $\Box$

 To see that the above may not hold for super-solutions, consider the classical heat equation $\Delta u-u_t=0$. If we take $(u(x),\lambda_1)$ to be the first eigenfunction, eigenvalue pair of $\Delta$, with $u>0$, and define $u(x,t)=u(x)$, we get that $\Delta u-u_t= -\lambda_1 u\le 0$ and $u=0$ on $\p\Om\times(0,\infty)$. It is well-known that $u\in C^\infty$and is a viscosity solution. Clearly, $u$ does not decay in $t$.

\appendix
 \section{Existence for the auxiliary elliptic problem and the Eigenvalue Problem}
 \vsp
 
 We begin with a version of the comparison principle that will be used in this section. Let $\Om\subset \IR^n$ be a bounded domain. We recall a result proven in \cite{BM4}.

\begin{lem}\label{Comp:BM4} Let $f_i:\Om\times\IR\rightarrow \IR$, $i=1,2$, be continuous as in \eqref{sec2.2}. Suppose that $u\in usc(\overline{\Om})$ and $v\in lsc(\overline{\Om})$ are solutions to
$$H(Du,D^2 u) \ge f_1(x,u(x))\;\;\;\mbox{and}\;\;\;H(Dv,D^2 v)  \le f_2(x,v(x)),\;\;\mbox{in $\Om$}.$$
If $\sup_\Om(u-v)>\sup_{\p\Om} (u-v)$ then there is a point $z\in \Om$ such that 
$$(u-v)(z)=\sup_\Om (u-v){\mbox{ and }}
f_1(z, u(z))\le f_2(z, v(z)).$$
\end{lem}
\begin{proof} A proof can be worked out as in Theorem 4.1 in [\cite{BM4}: Section 4]. \end{proof}

\begin{cor}\label{cmp}(Comparison Principle) Suppose that $s,\; t\in \IR$ are such that $|s|+| t | >0$, and $s\le t.$ 
Let $u\in usc(\overline{\Om})$ and $v\in lsc(\overline{\Om})$ satisfy  
$$H(Du,D^2 u) \ge t,\quad\mbox{and}\quad H(Dv,D^2 v) \le s \quad \mbox{in $\Om$.}$$
Then $u-v\le \sup_{\p\Om}(u-v).$ 
\end{cor}
\begin{proof} Consider $s<t.$ By taking $f_1=t$ and $f_2=s$, Lemma \ref{Comp:BM4} implies that $u-v\le \sup_{\p\Om}(u-v)$.

Assume now that $t=s$. We take $\tht>1$ if $t>0$, and $0<\tht<1$ if $t<0$. The function $u_\tht=\tht u$ solves
$H(Du_\tht, D^2u_\tht)=\tht^k H(Du, D^2u)\ge t\tht^k>s$. Thus, 
$$u_\tht-v\le \sup_{\p\Om}(u_\tht-v).$$
The conclusion follows by letting $\tht \rightarrow 1.$ 
\end{proof}

We obtain also the following quotient form of the comparison principle, see [\cite{BM4}: Theorem 1.2, see Sections 1 and 5]. 

\begin{lem}\label{Quo:BM4} Let $\lam>0$. Suppose that $u\in usc(\overline{\Om})$, $u>0$, and $v\in lsc(\overline{\Om})$, $v>0$ in $\overline{\Om}$, are solutions to
$$H(Du,D^2 u) +\lam u^k\ge 0 \;\;\;\mbox{and}\;\;\;H(Dv,D^2 v)+\lam v^k\le 0,\;\;\mbox{in $\Om$}.$$
Then $u/v\le \sup_{\p\Om} (u/v).$
\end{lem}

\vsp

 \subsection{Existence for Lemma \ref{thm1.2-3} }
 
  Let $\dl>0$  and $\tht\in \IR$. In this appendix we show existence of viscosity solutions to the following problems by using the Perron method.
\bea\label{Ap:eq1}
&&\mbox{(a)} \; H(Du, D^2u)=\dl,\;\mbox{in $\Om$,}\; u=\tht \mbox{ on } \p \Om,\;\mbox{and} \nonumber\\
&&\mbox{(b)} \; H(Du,D^2u)=-\dl,\;\mbox{in $\Om$,} \; u=\tht \mbox{ on } \p \Om.
\eea

We construct suitable sub-solutions and super-solutions. Corollary \ref{cmp} provides the necessary comparison principle.
Define
\be\label{diam}
d=\mbox{diam}(\Om).
\ee

Observe that for any $y\in \p\Om$, there is a $\rho>0$ and a $q\in \IR^n\setminus \Om$ such that 
\be\label{OB}
B_{\rho}(q)\subset \IR^N\setminus \Om\quad\text{and}\quad y\in \p\Om\cap \overline{B}_{\rho}(q).
\ee

\vsp
{\bf Sub and Super solutions to (\ref{Ap:eq1})(a):}  We note that, for any $\tht$, $w(x)=\tht$ is a super-solution of (\ref{Ap:eq1})(a). Our effort is to construct sub-solutions.
\vsp 
Let $y\in \p\Om$. With $d$ as in (\ref{diam}), and $\rho$ and $q_y$ as in (\ref{OB}), set $r=|x-q|.$ Define
$$v_y(x)=\tht+ E \left( \frac{1}{r^\al}-\frac{1}{\rho^\al} \right),\;\;\forall x\in \Om.$$
where $E>0$ and $\al>0$ are to be determined. Using (\ref{sec2.5}), we get, in $r\ge \rho$,
\bea\label{Ap-1}
H(Dv_y, D^2v_y)&=&E^kH\left( \frac{-\al}{r^{\al+1} } e,  \frac{-\al}{r^{\al+2} }(I-e\otimes e) +\frac{ \al(\al+1) }{ r^{\al+2} } e\otimes e   \right)   \nonumber \\
&=&   \frac{   (E\al)^k }{ r^{\al k+k+1} }  H\left( e,  (\al+2) e\otimes e-I   \right).
\eea  

\vsp
Setting $\Lam=\al+2$, and recalling (\ref{mM}) and Condition C(ii) in Section 1 (see (\ref{sec2.7})), 
$$\min_{|e|=1}H(e, \Lam e\otimes e-I)\ge -M(\Lam)>0,\quad\text{if $\Lam>\Lam_1$}.$$
Choose $\Lam>\Lam_1$ and $\al=\Lam-2.$ Next, observing that if $x\in \Om$ then $\rho\le r\le \rh+d$, (\ref{Ap-1}) yields in $\Om$,
\ben
H_k[v_y]\ge \frac{ (E\al)^k |M(\Lam)| }{(\rho+d)^{k\al+k+1} }>0.
\een
\vsp
We now select $E$ such that
$$ \frac{ (E\al)^k |M(\Lam)| }{(\rho+d)^{k\al+k+1} }\ge \dl.$$

With this choice, we obtain that
\ben
H(Dv_y, D^2v_y)\ge \dl ,\quad v_y(y)=\tht, \quad\mbox{and}\quad v_y\le \tht\;\;\mbox{on $\p\Om.$}
\een

For every $y\in \p\Om$, we have constructed a sub-solution $v_y$
that attains the boundary value $\tht$ at $y$. The Perron Method leads to a solution $v_y\le u\le w=\tht$ of
(\ref{Ap:eq1})(a). 

\vsp

{\bf Sub and Super solutions to (\ref{Ap:eq1})(b):} Observe that $v(x)=\tht$ is a sub-solution. Our effort is to construct super-solutions. 
\vsp
Let $y\in \p\Om$. With $d$ as in (\ref{diam}), and $\rho$ and $q$ as in (\ref{OB}), set $r=|x-q|.$ Define
$$w_y(x)=\tht+ E \left( \frac{1}{\rho^\al}-\frac{1}{r^\al} \right),\;\;\forall x\in \Om,$$
where $E>0$ and $\al>0$ are to determined. Using (\ref{sec2.5}), we get, in $r>0$,
\ben
H(Dw_y, D^2w_y)&=&E^k H\left( \frac{\al}{r^{\al+1} } e,  \frac{\al}{r^{\al+2} } (I-e\otimes e) - \frac{\al(\al+1) }{ r^{\al+2} } e\otimes e \right)\\
&=& \frac{ (E\al)^k}{ r^{\al k +k+1} } H( e, I- (\al+2) e\otimes e).
\een

Set $\Lam=\al+2$. Recalling (\ref{mM}) and Condition C(ii), we see that
$$\max_{|e|=1} H(e, I-\Lam e\otimes e)\le M(\Lam)<0,$$
if $\Lam>\Lam_1$. Choose $\Lam>\Lam_1$ and $\al>\Lam-2.$ Since, $\rho\le r\le \rho+d$, we see that
$$H(Dw_y, D^2w_y)= \frac{ (E\al)^k}{ r^{\al k +k+1} } H( e, I- (\al+2) e\otimes e)\le \frac{ (E\al)^k M(\Lam) }{ (\rho+d)^{\al k +k+1} }<0.$$ 
Choose $E>0$ such that
$$\frac{ (E\al)^k |M(\Lam)| }{ (\rho+d)^{\al k +k+1} }\ge \dl.$$

Thus,
$  H(Dw_y, D^2w_y)\le -\dl,\;\mbox{in $\Om$,}\quad \bar{w}_y(y)=\tht,\quad\mbox{and}\quad w_y\ge \tht\;\mbox{on $\p\Om.$}$\\
By the Perron method, there is a solution $u$ such that $\tht=v\le u\le w_y$.   $\Box$

Next, we discuss the results needed for Theorem \ref{thm:1.2}. We refer to the work \cite{BM4}. \\
Recall the hypothesis that $\Om\in C^2$. In \cite{BM4} a distinction is made between the cases $1\le \Lam_1<2$ and $\Lam_1\ge 2$. This is not required here.

\vsp

\subsection{Eigenvalue Problem}

 We show here that $\lam$ used in (\ref{thm1.2-6}) is bounded, see [\cite{BM4}: (1.10), Section 1]. Let $k\ge 1$ and $\dl>0$. Let the differential operator $H$ satisfy conditions A, B and C. Consider the problem of the existence of   a pair $\lam\in \IR$ and $u>0$ satisfying
\be\label{A.2-1}
H(Du, D^2u)+\lam u^k=0,\;\;\mbox{in $\Om$, and $u=\dl$ on $\p\Om$.}
\ee
 Define
$S=\{\lam\; :\; \mbox{Problem (\ref{A.2-1}) has a positive solution $u$} \}.$\\ 
It is shown in [\cite{BM4}: Theorem 1.5, Sections 1 and 8] that $S$ is an interval, and
$\lamo=\sup S<\infty.$\\
 This is shown in [\cite{BM4}: Theorem 1.7, Sections 1 and 9]. The proof uses domain monotonicity of $\lamo$. This is shown in [\cite{BM4}: Lemma 8.2, Section 8].

\NI Department of Mathematics, Western Kentucky University, Bowling Green, KY 42101, USA\\
\NI Dept. of Mathematics \& Computer Science, Rutgers University, Newark, NJ 07103, USA


\begin{thebibliography}{99}

\bibitem{AJK} G. Akagi, P. Juutinen and R. Kajikiya, \it Asymptotic behavior of viscosity solutions for a degenerate parabolic equation associated with the infinity-Laplacian, 
\rm Math.Ann. 343 (2009), no 4, 921-953.

\bibitem{BM1} T. Bhattacharya and L. Marazzi, {\it On the viscosity solution to a parabolic equation,} Annali di Matematica Pura ed Applicata, vol 194, no 5, 2014.
DOI:10.1007/s10231-014-0427-1

\bibitem{BM2} T. Bhattacharya and L. Marazzi, {\em On the viscosity solutions to Trudinger's equation},
Nonlinear Differential equations and applications (NoDEA), vol 22, no 5, 2015.
DOI:10.1007/s00030-015-0315-4

\bibitem{BM20} T. Bhattacharya and L. Marazzi, {\em Erratum to: On the viscosity solution to Trudinger's equation,} Nonlinear Differential equations and applications (NoDEA), vol 23, no 68, 2016.

\bibitem{BM3} T. Bhattacharya and L. Marazzi, {\em Asymptotics of viscosity solutions of some doubly nonlinear parabolic eqns.} J. of Evol. Eqns.  DOI:10.1007/s00026-015-0319-x

\bibitem{BM4} T. Bhattacharya and L. Marazzi, {\it On the viscosity solutions of eigenvalue problems for a class of nonlinear elliptic equations.} Advances in Calculus of Variations, vol 12, issue 4, 2019, 393-421.

\bibitem{BM5} T. Bhattacharya and L. Marazzi, {\it On the viscosity solution to a class of nonlinear degenerate parabolic differential equations.} Revista Mathematica Complutense 30, 621-656(2017).

\bibitem{CIL} M. G. Crandall, H. Ishii and P. L. Lions,\it User's guide to viscosity solutions of second order partial differential equations,\rm Bull. Amer. Math. Soc. 27(1992) 1-67.

\bibitem{ED0} E. DiBenedetto, {\it Degenerate Parabolic Equations} (Universitext), Springerl Verlag 1993

\bibitem{ED} E. DiBenedetto, {\it Partial Differential Equations}: Second Edition (Cornerstones), Birkhauser 2009

\bibitem{EDUV} E. DiBenedetto, U. Gianazza and V. Vespri, {\it Harnack's Inequality for Degenerate and Singular Parabolic Equations}, 2010 Monograph Preprint.

\bibitem{JL}  P. Juutinen and P. Lindqvist, \it Pointwise decay for the solutions of degenerate and singular parabolic equations, \rm Adv. Differential Equations 14(2009), no. 7-8, 663-684.

 \bibitem{LSU} O. A. Ladyzenskajia, N. A. Solonnikov and N. N. Ural'tzeva, 
{\it Linear and Quasilinear Equations of Parabolic Type,} Translations of Mathematical Monographs, 23, AMS,1967 

\bibitem{L} E. M. Landis, {\it Second Order Equations of Elliptic and Parabolic Type}, Translations of Mathematical Monographs 1998, AMS. 


\bibitem{TR} N. S. Trudinger, {\it Pointwise estimates and quasilinear parabolic equations,} Comm. Pure Appl. Math. 21, 205-226 (1968) 

\bibitem{Tych} A. Tychonoff,  {\em Theoremes d'unicite pour l'equation de la chaleur,}
Mat. Sb., 1935, Volume 42, Number 2, 199--216.

\end{thebibliography}
\end{document}